\documentclass[12pt,twoside]{article}
\def\version{20.7.2026}\def\users{}  %
\def\users{final-layout}  
\usepackage{times}
\usepackage{enumerate}
\usepackage{amsmath,amsfonts,amssymb,color}
\usepackage{mathrsfs} 
\usepackage{bm}
\usepackage{epsfig}
\usepackage{psfrag}
\usepackage{hyperref}
\usepackage{upgreek} 
\usepackage[format=hang, indention=0cm, labelformat=simple, labelsep=space, textformat=simple, justification=justified, singlelinecheck=true, labelfont=bf, textfont=sl,width=1.0\textwidth,figurename=Figure,skip=15pt]{caption}
\usepackage{cite}

\newtheorem{theorem}{Theorem}[section]

\newtheorem{definition}[theorem]{Definition}

\newtheorem{proposition}[theorem]{Proposition}

\newtheorem{remark}[theorem]{Remark}

\newcommand{\lineunder}[2]{\LU{\begin{array}[t]{c}\underbrace{#1}\vspace*{.5em}\end{array}}{\mbox{\footnotesize\rm #2}}}

\newcommand{\LU}[2]{\begin{array}[t]{c}#1\vspace*{-1em}\\_{#2}\end{array}}

\newcommand{\Fezero}{\FF_{\hspace*{-.1em}\mathrm e,0}^{}}

\numberwithin{equation}{section}
\usepackage{ifthen}
\ifthenelse{\equal{\users}{final-layout}}{
\markboth{T.Roub\'\i\v cek}{Stefan problem with complete melting}
}{
\usepackage{fancyhdr}
\pagestyle{fancy}
\headheight=28pt\headwidth=16.5cm
\definecolor{gray}{gray}{0.5}
\rhead{\color{gray}Stefan problem \\
 T.Roub\'\i\v cek}
\chead{}
\lhead{Version \version, file: \jobname.tex \\
compiled:
\number\day.\number\month.\number\year\ at
\the\hour:\ifnum\minute<10 0\fi\the\minute\ h }
}

\ifthenelse{\equal{\users}{final-layout}}{}{
\usepackage[notcite,notref,color]{showkeys}
\definecolor{labelkey}{rgb}{1.,.2,0.}

}

\newcount\hour \newcount\minute
\hour=\time
\divide \hour by 60
\minute=\time
\loop \ifnum \minute > 59 \advance \minute by -60 \repeat

\usepackage[normalem]{ulem}
\definecolor{brown}{rgb}{0.5,0,0}
\ifthenelse{\equal{\users}{final-layout}}{

    \newcommand{\DELETE}[1]{}
    
    \newcommand{\COMMENT}[1]{}
    
    \newcommand{\TINY}[1]{}
    \newcommand{\MARGINOTE}[1]{}
}
{

 \newcommand{\DELETE}[1]{{\color{brown}\sout{#1}\color{black}}}
 
 \newcommand{\COMMENT}[1]{{\color{red}\uuline{#1}\color{black}}}
 
 \newcommand{\TINY}[1]{{\tiny#1}}
 \newcommand{\MARGINOTE}[1]{\marginpar{\color{red}\tiny\texttt{#1}}}
}

\newcommand{\Item}[3][0.05]{\parbox[t]{#1\textwidth}{#2}\hfill%
      \parbox[t]{\dimexpr\textwidth-#1\textwidth}{#3}\vspace*{.8mm}}

\newcommand{\wt}[1]{\mathchoice
     {\text{\small$\widetilde{\text{\normalsize$#1$}}\hspace*{.03em}$}}
                    {\text{\small$\widetilde{\text{\normalsize$#1$}}$}}
                    {\widetilde{#1\hspace*{-.02em}}\hspace*{.03em}}
                    {\widetilde{#1}}}
\newcommand{\ITEM}[2]{\parbox[t]{.05\textwidth}{{\rm #1}}\hfill\parbox[t]{.95\textwidth}{#2}\vspace*{.5em}\\}
\def\R{{\mathbb R}}
\newcommand{\barOmega}{\hspace*{.2em}{\wb{\hspace*{-.2em}\varOmega}}}
\newcommand\bbC{\mathbb C}
\newcommand\bbD{\mathbb D}
\renewcommand\d{\mathrm d}
\newcommand{\nablaS}{\nabla_{\scriptscriptstyle\textrm{\hspace*{-.3em}S}}^{}}
\newcommand{\divS}{\mathrm{div}_{\scriptscriptstyle\textrm{\hspace*{-.1em}S}}^{}}
\newcommand{\FF}{\bm{F}}
\newcommand{\Fe}{\FF_{\hspace*{-.15em}\mathrm e^{^{^{}}}}}
\newcommand{\Fee}{\FF_{\hspace*{-.2em}\mathrm e}}
\def\FedevTAU{{\bm{F}_{\!{\rm e},\TAU}^\ast}}
\newcommand{\Fedev}{{\bm{F}_{\hspace*{-.05em}\mathrm e}^\ast}}
\newcommand{\Feedev}{{\bm{F}_{\hspace*{-.05em}\mathrm e}^\ast}}
\def\tildeFedevTAU{\wt{\bm{F}}_{\!\!{\rm e},\TAU}^\ast}

\newcommand{\Fetop}{\FF_{\hspace*{-.2em}\mathrm e}^\top}
\newcommand{\Fedevzero}{\FF_{\hspace*{-.2em}\mathrm e,0}^\ast}
\newcommand{\Fp}{\FF_{\hspace*{-.15em}\mathrm p^{^{^{}}}}}
\newcommand{\Fpp}{\FF_{\hspace*{-.15em}\mathrm p}}
\newcommand{\Lp}{\bm{L}_{\hspace*{-.0em}\mathrm p^{^{^{}}}}}
\newcommand{\Le}{\bm{L}_{\hspace*{-.0em}\mathrm e^{^{^{}}}}}
\newcommand{\eq}[1]{(\ref{#1})}
\newcommand{\zetap}{\zeta_{\rm vp}}
\newcommand\pdt[1]{\frac{\partial{#1}}{\partial t}} 
\newcommand\DT[1]{\mathchoice
                 {{\buildrel{\hspace*{.1em}\text{\Large.}}\over{#1}}}
                 {{\buildrel{\hspace*{.1em}\text{\large.}}\over{#1}}}
                 {{\buildrel{\hspace*{.1em}\text{\large.}}\over{#1}}}
                 {{\buildrel{\hspace*{.1em}\text{\large.}}\over{#1}}}}
\newcommand{\Frac}[2]{\mathchoice{\text{\small$\frac{#1}{#2}$}}
                                 {\text{\large$\frac{#1}{#2}$}}
                                 {\text{\large$\frac{#1}{#2}$}}
                                 {\text{\large$\frac{#1}{#2}$}}}
\newcommand{\wb}[1]{\mathchoice{\hspace*{-.09em}\text{\large$\hspace*{.09em}\bar{\text{\normalsize$#1$}}\hspace*{.05em}$}\hspace*{-.05em}}
  {\hspace*{-.09em}\text{\large$\hspace*{.09em}\bar{\text{\normalsize$#1$}}\hspace*{.05em}$}\hspace*{-.05em}}
  {\text{\normalsize$\hspace*{.08em}\bar{\text{\scriptsize$#1$}}\hspace*{.06em}$}}
  {\text{\small$\bar{\text{\tiny$#1$}}$}}}
\def\Vdots{\!\mbox{\setlength{\unitlength}{1em}
\begin{picture}(0,0)
\put(-.07,0){.}
\put(-.07,.3){.}
\put(-.07,.6){.}
\end{picture}\hspace*{.2em}}}
\def\ENT{\mbox{\small{${\mathscr U}$}}}
\def\Eng{e}
\def\ENG{\mbox{\small{${\mathscr E}$}}}
\def\Ent{u}
\def\DD{\bm{D}}
\def\TT{\bm{T}}
\def\MM{\bm{M}}
\def\GG{\bm{G}}
\def\LL{\bm{L}}
\def\yy{\bm{y}}
\def\xx{\bm{x}}
\def\tt{\bm{t}}
\def\vv{\bm{v}}
\def\vvk{\bm{v}_\tau^k}
\def\pp{\bm{p}}
\def\qq{\bm{q}}
\def\nn{\bm{n}}
\def\Cdot{{\cdot}}
\def\bbI{\mathbb{I}}
\def\GM{M}
\def\NU{R}
\newcommand{\COUPLING}{\gamma}
\newcommand{\rhoR}{\varrho_\text{\sc r}}
\def\boldT{{\mathscr{T}\hspace*{-1.08em}\mathscr{T}\hspace*{-.1em}}}
\def\boldM{{\mathscr{M}\hspace*{-1.25em}\mathscr{M}\hspace*{-.1em}}}

\def\boldG{{\mathscr{G}\hspace*{-.85em}\mathscr{G}\hspace*{-.1em}}}
\def\ADI{\mbox{\Large{$_{\mathcal P}$}}}
\newcommand{\GRAVITY}{{\bm g}}
\newcommand{\strain}{{\boldsymbol{\varepsilon}}}
\newcommand{\Colon}{\hspace{-.15em}:\hspace{-.15em}}
\newcommand{\Rdev}{\R_{\rm dev}^{3\times3}}
\newcommand{\Rsym}{\R_{\rm sym}^{3\times3}}
\newcommand{\Rtr}{\R^{3\times3}_{\rm dev}}
\def\HYPER{\mathcal{D}}

\def\LpTAU{\LL_{{\rm p},\TAU}}

\def\TAU{\tau}
\def\DEV{{\rm dev}}
\def\Ld{\bm{L}_{\rm e}^{\!*\,}}
\def\LdTAUk{{\bm{L}^{k}_{{\rm e},\TAU}}}
\def\TAU{\tau}

\def\LdTAUk{\LL^{\!*\,k\ }_{{\rm e},\TAU}}
\def\LdTAUm{\LL^{\!*\,m\ }_{{\rm e},\TAU}}
\def\opT{\wb{\pp}_\TAU}
\def\ovT{\wb{\vv}_\TAU}
\def\osT{\wb{\sigma}_\TAU}
\def\orT{\wb{\varrho}_\TAU}
\def\utT{\underline{\theta}_\TAU}
\def\otT{\wb{\theta}_\TAU}
\def\uJT{\underline{J}_\TAU}
\def\oJT{\overline{J}_\TAU}
\def\oxT{\wb{\chi}_\TAU}
\def\oeT{\wb{\Eng}_\TAU}
\def\uxT{\underline{\chi}_\TAU}
\def\ueT{\underline{\Eng}_\TAU}
\def\uuT{\underline{\Ent}_\TAU}
\def\urT{\underline{\varrho}_\TAU}

\def\uTT{\underline{\bm T\!}_\TAU}
\def\ouT{\wb{\Ent}_\TAU}
\def\oDTzero{\vspace{.15em}\wb{\vspace{-.15em}\DD}_{\!0,\TAU}}
\def\oDTone{\vspace{.15em}\wb{\vspace{-.15em}\DD}_{\!1,\TAU}}
\def\oLdT{\wb{\LL}_{{\rm e},\TAU}^{\!*\ }}
\def\uFedevT{{\underline{\FF}_{\!{\rm e},\TAU}^\ast}}
\def\oFedevT{\wb{\FF}_{\!{\rm e},\TAU}^\ast}
\def\oLpT{\wb{\LL}_{{\rm p},\TAU}}
\def\ALPHA{\alpha}
\def\EXP{\mu}
\def\THRESHOLD{E}

\begin{document}

\allowdisplaybreaks

\newcommand{\subjclass}[1]{\bigskip\noindent\emph{2020 Mathematics Subject Classification:}\enspace#1}
\newcommand{\keywords}[1]{\noindent\emph{Keywords:}\enspace#1}

\title{The Stefan problem for complete melting of finitely strained solids into viscoelastic fluids
}

\author{Tom\'{a}\v{s} Roub\'\i\v{c}ek\\
Mathematical Institute, Charles University
\\ Sokolovsk\'a 83, CZ-186~75~Praha~8, Czechia\\
tomas.roubicek@mff.cuni.cz\\
and\\
Institute of Thermomechanics, Czech Academy of Sciences\\Dolej\v skova~5,
CZ-182~08 Praha 8, Czechia}

\date{}

\maketitle

\begin{abstract}
The compressible fluid-solid interaction (FSI) with a thermomechanical phase transition is formulated at large strains within the Eulerian frame. For the deviatoric part, the Jeffreys (also called anti-Zener) rheology with an additional viscosity is adopted. The core philosophy governing the mechanical solid-liquid transition is that the viscous (or viscoplastic) response is temperature-dependent and may fully degenerate to a viscoelastic fluid during thawing, so that there is no elastic response on the shear distortion. This behavior enables the free flow of the fluid, its subsequent freezing into a new configuration, and potential re-melting back into a fluid, allowing such cycles to repeat indefinitely. The classical Stefan problem, associated with the latent heat of the first-order (thawing-freezing) phase transition, is augmented by incorporating kinetic overheating and undercooling. The analysis by a time discretization with an appropriate truncation is applied to a higher-gradient modification of the original formulation, utilizing the concept of multipolar nonsimple continua.


\medskip

\end{abstract}

\section{Introduction}\label{sec-intro}

Since the seminal talk \cite{Stef89PTW} of Josef Stefan 
on the meeting of the Imperial Academy of Sciences in Vienna
on May 21, 1889, the so-called {\it Stefan problem}
experienced a vast development during the whole 20th century.
In his talk, he formulated the one-dimensional
free-boundary problem arising in melting of ice to water with a
specific latent heat in the frame of mere heat transfer.
This solid-liquid phase transition was long time considered rather 
as a mere thermal problem, enhanced with various
fine phenomena (like supercolling/heating or curvature-dependent
interface tension etc.), possibly in a media with prescribed movement (as
in continuous casting), or in incompressible, merely liquid environment
\cite{DiBOLe93TDCC,FukKen05SPCG,Nedo97CSLP,RodUrb98SBSP,RodUrb02TDCS,XuShi97SPCJ}
possibly allowing only ``phase-wise'' incompressibility \cite{PruShi12WPIT}.
The temperature-dependent shear viscosity (but not allowing
the complete melting) in a linearized semi-compressible model was
considered in \cite{Roub23SPTC}. This sort of problems has attained 
also attention in engineering (not-analytically supported) research,
as e.g.\ \cite{Myer25TPSP,MyHeCS20SPVT,WVKT23USPK}.

The prominent application (and the original motivation of the aforementioned
Stefan's talk) is the ice-water transition and glaciology in a wider sense. Ice
is conventionally modelled as a Maxwellian solid
with a Glen's (phenomenological) law \cite{Glen55CPI}, i.e.\ mechanically
as a non-Newtonian shear-thinning viscoelastic (or sometimes viscous
incompressible) fluid, while water is a Newtonian fluid naturally without any
elastic shear response. If compressible (as considered in this paper),
the elastic volumetric response is presented both in the solid and the liquid
phases. As a result, pressure (longitudinal) waves can propagate
in both phases (with the velocity in ice and in water about 3.9 and
1.5\,km/s, respectively) while shear waves can propagate only in the
solid phase (with the velocity in ice about 2\,km/s).
As a result, since the shear waves are completely blocked in the
liquid phase, the elastic waves are reflected and refracted on the
phase interfaces. This experimentally justifies the phenomenon
that the shear elasticity should be completely eliminated in the liquid phase,
cf.~Figure~\ref{fig-rheology-Jeff-St-thermo} below. Of course, attenuation
of waves due to viscosity in both the solid and the liquid phases is
also presented.

Yet, such model does not seem easily feasible for rigorous
analysis for three main reasons:\smallskip\\
\Item{(i)}{The aforementioned shear elastic response 
drops to zero when the medium is
melted, i.e.\ a transition from a regular
elasticity to degenerate elasticity (and therefore the elastic distortion
rate becomes uncontrolled within this transition), which corrupts
usual shear inelastic distortion rate bounds.
.}\smallskip\\
\Item{(ii)}{There are naturally
some discontinuities across the phase-transition temperature.
Beside the aforementioned elastic response
which drops to zero possibly even discontinuously, there are
also in some transport coefficients which are typically jumping
during the solid-liquid transition, typically the thermal conductivity
coefficients or (if involved) some diffusivity or electrical
conductivity, etc. For example, thermal conductivity of water is about
0.6 W/(m${\cdot}$K) while the thermal conductivity
of ice at the freezing temperature is 2.2 W/(m${\cdot}$K), so there
is a substantial discontinuity across the phase transition.
}\smallskip\\
\Item{(iii)}{many nonlinearities in 
  visco-elastodynamics at large strains makes analytical troubles, which
  leads to necessity of various very weak solution concepts, if 
  a concept of nonsimple materials involving some higher gradients would not
  be used.}
For these reasons, the basic philosophy is to be modified without
destroying the ability for complete melting towards a viscoelastic
fluid with unlimited repetition of freezing-melting cycles and
respecting typical discontinuity of some parameters
as viscosity or thermal conductivity. Specifically, to cope with (i), we will
respectively achieve it by including the additional viscosity into the
elasticity which makes it ``parabolic'' of the Kelvin-Voigt type even
if the Maxwellian viscoplastic element drops suddenly to zero within
melting. This will be performed by involving the kinetic super-cooling/heating
(also referred as undercooling and overheating)
as in \cite{ColSpr95PFMZ,Visi85SPPH,Visi85SSEP}, which also helps to cope
with (ii) beside reflecting some physically relevant phenomena, 
cf.\ also \cite{Visi86SPKC} for a wider discussion.
Eventually, to overcome (iii),
we exploit the concept of nonsimple, multipolar materials by considering a
gradient enhancement of the viscosities in a sophisticated way which still
allows for complete melting.

As another novelty in contrast to \cite{Roub23SPTC} where a semi-compressible
linearized convected model was considered, we consider a fully
compressible finitely-strained medium.

For readers' convenience, let us summarize the basic notation used in what
follows:
\vspace*{-.9em}
\begin{center}
\fbox{
\begin{minipage}[t]{15em}\small\smallskip
$\yy$ deformation,\\
  $\vv$ velocity,\\
$\varrho$ mass density,\\
$\pp=\varrho\vv$ the linear momentum,\\
$\theta$ (absolute) temperature,\\
$\chi$ phase fraction,\\
$\Fe$ elastic distortion,\\
  $\Fedev=\Fe/J^{1/3}$ the isochoric part of $\Fe$,\\
$\GRAVITY$ gravity acceleration,\\
$\HYPER_0$, $\HYPER_1$ the hyper-viscosity coefficients,\\
$I=[0,T]$ a time interval, $T>0$,\\
$\psi:\R^{3\times 3}\to\R$ free energy,\\[.1em]
$\eta=-\psi_\theta'$ entropy,\\
$(\cdot)'$ (partial) derivative of a  mapping,\\[.1em]
$(\cdot)\!\DT{^{\,}}$ convective time derivative,
\end{minipage}
\begin{minipage}[t]{20em}\small\smallskip
$\TT$ Cauchy stress,\\
$\DD$ dissipative stress,\\
$\MM$ Mandel's stress,\\
$\Lp$ inelastic distortion rate,\\
$\Le$ deviatoric part of elastic distortion rate,\\
$\kappa$ the thermal conductivity,\\
$\bbI$ the identity matrix,\\
tr$(\cdot)$ trace of a matrix,\\
dev$(\cdot)$ deviatoric part of a matrix, dev$E:=E{-}\frac{{\rm tr}E}3\bbI$,\\
$\NU$ relaxation time for superheating/cooling,\\
$\R_{\rm sym}^{3\times3}$ set of symmetric matrices,\\
$\R_{\rm dev}^{3\times3}=\{A\in\R^{3\times3};\ {\rm tr}A=0\}$,\\
``$\:\Cdot\:$'', ``$\:\Colon\:$'' scalar products of vectors or matrices,\\ 
``$\,\Vdots\,$'' scalar products 3rd-order tensors,\\ 
$\tau>0$ a time step for discretization.
\smallskip \end{minipage}
}\end{center}
\nopagebreak
\vspace{-.9em}
\nopagebreak
\begin{center}
{\small\sl Table\,1.\ }
{\small
Summary of the basic notation used. 
}
\end{center}

The plan of this paper is first to discus the mere Stefan problem
without any mechanical context in 
Section~\ref{sec-Stefan}, including the kinetic superheating/cooling.
Then the mechanical context at large strains in Eulerian frames
is involved in Section~\ref{sec-thermo}. Eventually, the analysis in
a variant with enhanced (multipolar) viscosities is presented by a
partly decoupled time discretization in Section~\ref{sec-anal}.

\section{Basic philosophy, the Stefan problem, and a phase fraction.}\label{sec-Stefan}

Considering the nonlinear heat equation
$c(\theta)\pdt{}\theta={\rm div}(\kappa(\theta)\nabla\theta)$ with the heat
capacity $c=c(\theta)$ and the thermal conductivity $\kappa=\kappa(\theta)$
dependent on temperature $\theta$, the basic philosophy of treating the
Stefan problem is the transformation into the equation
\begin{align}
  \pdt{}\ENT(\theta)+{\rm div}\,\qq=0
\ \ \text{ with }\ \qq=-\kappa(\theta)\nabla\theta\,,
\label{heat-enthalpy}\end{align}
where the ``enthalpy''
(or rather the {\it thermal internal energy}) $\ENT$ is a primitive function (i.e.\ the
antiderivative)
of $c$ and $\qq$ is the heat flux governed by the (phenomenological) Fourier
law. This allows us to consider a Dirac function contributing to $c$ and 
supported at a phase-transition temperature $\theta_\text{\sc pt}$ 
with the magnitude describing the latent heat. This results to a jump in $\ENT$.
Then, \eq{heat-enthalpy} is the so-called enthalpy formulation of the Stefan
problems, cf.\ Figure~\ref{fig-Stefan}-left. The thermodynamics behind
\eq{heat-enthalpy} departs from the free energy $\psi=\psi(\theta)$ which
determines the entropy $\eta=-\psi'(\theta)$ governed by the entropy
equation
\begin{align}\nonumber\\[-2.7em]
\pdt{\eta}=-({\rm div}\,\qq)/\theta\,.
\label{heat-entropy}\end{align}
In fact, \eq{heat-enthalpy} follows from \eq{heat-entropy} with 
the internal energy $\ENT$ given by the Gibbs relation
$\ENT(\theta)=\psi(\theta)-\theta\psi'(\theta)$. Then
the heat capacity is $c(\theta)=\ENT'(\theta)=-\theta\psi''(\theta)$.
Notably, integrating \eq{heat-entropy} over a spatial domain $\varOmega$ and
assuming the thermal isolation (i.e.\ adiabatic system) with the
normal heat flux $\nn\Cdot\qq=0$ on the boundary of $\varOmega$,
we obtain the total entropy balance $\frac{\d}{\d t}\int_\varOmega\eta\,\d\xx
=\int_\varOmega\kappa(\theta)|\nabla\theta|^2/\theta^2\,\d\xx$, which reveals
the (non-negative) entropy production rate and the 2nd law of thermodynamics
if $\kappa(\theta)\ge0$.

The aforementioned possible discontinuity of $\kappa$ at the phase-transition
temperature $\theta_\text{\sc pt}$ can be handled by so-called
Kirchhoff transformation, i.e.\ by inventing
the primitive function $K$ of $\kappa$ and then using the transformed
temperature $\vartheta=K(\theta)$. Thus $\theta=K^{-1}(\vartheta)$ and this heat
equation transforms to $\pdt{}[\ENT{\circ}K^{-1}](\vartheta)=\Delta\vartheta$.
Yet, this latter transformation would be complicated if $\kappa$
were depending also on some other variables and, mainly, would
not cope with other discontinuities, as in viscous moduli depicted in
Figure~\ref{fig-Stefan}-right below.

This is why we will cope with such discontinuities by another way, namely the
aforementioned kinetic undercooling and overheating.
Moreover, we will use the ``enthalpy'' transformation in a fully
thermomechanical concept with $\ENT$ playing the role of the thermal energy
as a part of the internal energy. 

The problem of discontinuities of certain parameters
across the phase-transition temperature $\theta_\text{\sc pt}$ (as the Maxwellien-type
viscosity in Figure~\ref{fig-Stefan} below or the heat conductivity etc.) is analytically
troublesome. This suggests to make their dependence on some other variable
than on temperature.
In view of Figure~\ref{fig-Stefan}, a first choice would be the heat internal
energy $\Ent$ but it is conceptually doubtful to make them dependent on
an extensive
variable like $\Ent$. More suitably, we invent a {\it phase fraction}
(in a position of a {\it phase field})   
as an intensive variable $\chi$ ranging the interval $[0,1]$ with the
meaning $\chi=0$ for solid and  $\chi=1$ for liquid.

We can split the discontinuous nonlinearity $\ENT$ in \eq{heat-enthalpy}
into a its continuous part, denoted by $\wt\ENT$, and the rest as
a step-wise function. Specifically, it gives
\begin{align}
  \pdt{\Ent}={\rm div}(\kappa(\theta)\nabla\theta)\,,
\ \ \text{ where }\ \ 
\Ent=\wt\ENT(\theta)+
L\chi
\ \ \text{ with }\ \ \chi\in H\big(\theta{-}\theta_\text{\sc pt}\big)\,,
\label{heat-enthalpy+}\end{align}
where $H$ denotes the set-valued Heaviside function, i.e.\
$H(\cdot)=0$ on $(-\infty,0)$ and $H(\cdot)=0$ on $(0,+\infty)$
while, at 0, it equals to the interval $[0,1]$. 
In fact, \eq{heat-enthalpy+} can again be derived from a concave free energy
which is now only piecewise $C^1$, namely
\begin{align}
   &\qquad\qquad\psi(\theta)=\widetilde\psi(\theta)
-\begin{cases}\qquad 0
&\text{if }\ \theta\le\theta_\text{\sc pt}\,,
\\[-.1em]\
L\,\big(\,\theta/\theta_\text{\sc pt}-1\big)
&\text{if }\ \theta>\theta_\text{\sc pt}\,.
    \end{cases}
\label{split-phi}
\end{align}
The continuous part $\wt\ENT$ in \eq{heat-enthalpy+} is now
given by $\wt\ENT(\theta)=\widetilde\psi(\theta)-\theta\widetilde\psi'(\theta)$
while the last term in \eq{split-phi} indeed contributes as
$LH(\theta{-}\theta_\text{\sc pt})$ into
$\ENT(\theta)=\psi(\theta)-\theta\psi'(\theta)$.

Let us illustrate usage of $\chi$ for treating  a
discontinuity of the thermal conductivity $\kappa=\kappa(\theta)$
by distinguishing $\kappa$ for the solid and for the
liquid phases, denoting $\kappa_\text{\sc s}=\kappa_\text{\sc s}(\theta)$
for $\theta\le\theta_\text{\sc pt}$ and 
$\kappa_\text{\sc l}=\kappa_\text{\sc l}(\theta)$ for $\theta\ge\theta_\text{\sc pt}$,
respectively. Both $\kappa_\text{\sc s}$ and $\kappa_\text{\sc l}$ should be assumed
continuous on $[0,\theta_\text{\sc pt}]$ and on $[\theta_\text{\sc pt},+\infty)$,
respectively. We can extend them continuously as 
$\kappa_\text{\sc s}(\theta)=\kappa_\text{\sc s}(\theta_\text{\sc pt})$ for
$\theta\ge\theta_\text{\sc pt}$ and
$\kappa_\text{\sc l}(\theta)=\kappa_\text{\sc l}(\theta_\text{\sc pt})$ for
$\theta\le\theta_\text{\sc pt}$. Then we can define the resulted modulus
$\kappa=\kappa(\theta,\chi)$ at the phase transition as a convex combination
\begin{align}
\kappa(\theta,\chi)=(1{-}\chi)\kappa_\text{\sc s}(\theta)+\chi \kappa_\text{\sc l}(\theta)\,,
\label{M-convex-combination}\end{align}
relying on that $\chi\in H(\theta)$.

The basic Stefan problem can be enriched by various fine effects, in
particular towards {\it supercooling}\index{supercooling/heating}
(i.e.\ the effect of allowing the liquid to remain still in the
liquid phase even slightly below its normal its freezing point under
some conditions) and opposite process called {\it superheating}.
There can be various mechanisms to consider, both static (surface tension,
capillarity) and kinetic \cite{Cagi85STSS,Cagi89SHSM,Gurt94TSSE,Visi86SPKC}.
The kinetic supercooling and superheating, sometimes referred as undercooling
and overheating \cite{ColSpr95PFMZ,Visi85SPPH,Visi85SSEP}, reflects the phenomenon
that melting or solidification proceed with a speed which may increase
with the difference between the actual temperature and the equilibrium
(=\,phase transformation) temperature \cite{Chal64PS}. This enrichment
(also called a {\it relaxed Stefan problem}\index{Stefan problem!relaxed})
will also bring analytical benefits especially in the context of mechanical
coupling considered in the next Section~\ref{sec-thermo}.

To involve these kinetic supercooling\,/\,heating effects,
we will now slightly modify the model by  ``relaxing'' the last inclusion in
\eqref{heat-enthalpy+} written equivalently as
$H^{-1}(\chi)\ni\theta-\theta_{\rm pt}$. First, we write more generally (but still
equivalently) as $H^{-1}(\chi)\ni \Upsilon(\theta{-}\theta_\text{\sc pt})$
with $\Upsilon$ some continuous increasing function with $\Upsilon(0)=0$.
Then, for a (typically small) relaxation time $\NU>0$, we consider
\begin{align}
  \NU\pdt\chi+H^{-1}(\chi)\ni
\Upsilon(\theta{-}\theta_\text{\sc pt})\,.
\label{relaxed}\end{align}
For $\NU=0$, we obtain exactly the Stefan model \eq{heat-enthalpy+}.
In literature, usually $\Upsilon(\theta)=\theta$.
Such a model was combined with a Cattaneo modification of the Fourier law or
generalized in \cite{ColRec02CSPP,Recu23CRSP,RecZan26AASP}.
Yet, here we consider a general $\Upsilon$ with some sublinear growth,
cf.\ \eq{ass-Upsilon} below, which will be needed for the estimate
\eq{EUL-L-est-DT-chi-disc} which is further needed
for \eq{grad-chi} and then for the strong convergence of $\chi$'s.
Even a slightly more general dynamics for $\chi$ was considered in
\cite{Recu23CRSP,RecZan26AASP,Visi01MPR}.

\begin{remark}[``Enhanced'' temperature $\theta{+}\chi$.]\label{rem-enhanced}
\upshape
  The  jointly continuous $\kappa$ from \eq{M-convex-combination} can be
  restricted on the graph of the relation $\chi\in H(\theta)$
 in $[0,+\infty)\times[0,1]$,
i.e.\ ${\rm Gr}(H):=\{(\theta,\chi);\ \chi\in H(\theta)\}$. We obtain
\begin{align}\nonumber\\[-3.em]
  \kappa\big|_{{\rm Gr}(H)}^{}(\theta,\chi)=\begin{cases}\hspace*{4em}\kappa_\text{\sc s}(\theta)&\text{if }\
  \theta<\theta_\text{\sc pt}\,,\hspace{2em}\text{\small\sf (solid)}
 \\[-.3em]
      (1{-}\chi)\kappa_\text{\sc s}(\theta_\text{\sc pt})+
\chi \kappa_\text{\sc l}(\theta_\text{\sc pt})
\!\!&\text{if }\ \theta=\theta_\text{\sc pt}
\,,\hspace*{2em}\text{\small\sf (mushy)}
 \\[-.3em]\hspace*{4em}\kappa_\text{\sc L}(\theta)&\text{if }\
 \theta>\theta_\text{\sc pt}
 \,.\hspace{2em}\text{\small\sf (fluid)}
    \end{cases}
\label{M-convex-combination+}\end{align}
When the monotone graph ${\rm Gr}(H)$ is parameterized by
$\wt\theta:=\theta+\chi$, we can also write
\begin{align}
\kappa\big|_{{\rm Gr}(H)}^{}(\wt\theta)
=\begin{cases}\hspace*{4em}\kappa_\text{\sc s}(\theta)&\text{if }\
\wt\theta<\theta_\text{\sc pt}\,,
 \\[-.1em]
      (1{-}\chi)\kappa_\text{\sc s}(\theta_\text{\sc pt})+
\chi \kappa_\text{\sc l}(\theta_\text{\sc pt})
\!\!&\text{if }\ 
\theta_\text{\sc pt}\le\wt\theta\le\theta_\text{\sc pt}{+}1
\text{ with $\chi=\wt\theta-\theta_\text{\sc pt}$}\,,
 \\[-.1em]\hspace*{4em}\kappa_\text{\sc L}(\theta)&\text{if }\
\wt\theta>\theta_\text{\sc pt}{+1}\,.
    \end{cases}
\label{M-convex-combination++}\end{align}
The chosen parameterization $\wt\theta:=\theta+\chi$ plays a role of an
artificial ``enhanced'' temperature serving for a ``continualization'' of
discontinuous coefficients, as illustrated on Figure~\ref{fig-enhanced-temperature}
below for the creep viscous modulus. This would need to consider $\chi$ formally
in the physical unit K and then $L$ in Pa/K.
\end{remark}

\section{The thermo-visco-elastodynamics in Jeffreys-Stokes rheology}\label{sec-thermo}

To present the Stefan problem in the thermomechanical context
at large strains, we briefly recall the basic kinematics
and thermodynamics of the visco-eleastodynamics in Eulerian frame.
Cf.\ the continuum-mechanics textbooks as e.g.\ \cite{GuFrAn10MTC,Mart19PCM}.
We consider a fixed bounded domain $\varOmega\in\R^3$
whose (referential) points will be denoted by $\bm X$.
The basic ingredient is the deformation $\yy:\varOmega\to\R^3$
which sends $\bm X$ to $\xx=\yy(\bm X)$ in the actually deformed domain
(but with the same boundary), so again $\xx\in\varOmega$. This
determines the Eulerian deformation gradient $\FF=\nabla\yy$. When $\yy$ evolves in
time, it determines the Eulerian velocity $\vv$ the 
and then also the {\it convective} (or material) {\it time derivative}
time derivative $\pdt{}\cdot+(\vv\Cdot\nabla)\cdot$,
denoted by $(\cdot)\!\DT{^{}}$ and to be applied to scalars or,
component-wise, to vectors or tensors. More precisely,
$\vv$ results from $\pdt{}\yy$ composed by the inverse $\yy^{-1}(t,\cdot)$
assumed to exist, so that $\vv=\vv(t,\xx)$. Also $\FF=\FF(t,\xx)$ is actually defined as
$\nabla_{\bm X}\yy$ composed by $\yy^{-1}(t,\cdot)$. By the chain-rule calculus,
we have the {\it transport equation-and-evolution for the
deformation gradient} and its determinant and its inverse as
\begin{align}\nonumber\\[-2.9em]
\DT\FF=(\nabla\vv)\FF\,, \ \ \ \ \ \ \ \
\DT{\overline{\det\FF}}=(\det\FF){\rm div}\,\vv\,,
\ \ \ \text{ and }\ \ \ \
\DT{\overline{\!\!\!\bigg(\frac1{\det\FF}\bigg)\!\!\!}}\ 
=-\frac{{\rm div}\,\vv}{\det\FF}\,.
\label{ultimate}\end{align}
The mass density (in kg/m$^3$) transport and the
``mass sparsity'' as the inverse mass density $1/\varrho$ transport write: 
\begin{align}\nonumber\\[-2.9em]
\DT\varrho=-\varrho\,{\rm div}\,\vv\ \ \ \ \text{ and }\ \ \ \ 
\DT{\overline{\!\!\bigg(\frac1\varrho\bigg)\!\!}}\
=\frac{{\rm div}\,\vv}{\varrho}\,.
\label{cont-eq+}\end{align}
The former equation in \eq{cont-eq+} called the {\it continuity equation}
equivalently writes $\pdt{}\varrho+{\rm div}(\varrho\vv)=0$ (expressing
that the conservation of mass) and  ensures the transport of the momentum
$\pp=\varrho\vv$: 
\begin{align}\label{inertial}
\pdt\pp+\text{\rm div}(\pp
{\otimes} \bm{v})=\varrho\DT\vv\,.
\end{align}

Further modelling ingredient is introducing an {\it inelastic} 
(or plastic) {\it distortion} tensor $\Fp$ in the position
of an internal variable having the interpretation of a transformation of the reference
configuration into an intermediate stress-free configuration.
A con\-ventional large-strain inelasticity is then based on the
Kr\"oner-Lee-Liu \cite{Kron60AKVE,LeeLiu67FSEP} {\it multi\-pli\-cative
decomposition}
of the deformation gradient $\FF$ to the inelastic and
{\it elastic distortion}s
\begin{align}\label{KLL}
\FF=\Fe\Fp\,\ \ \ \text{ with }\ \ \ \FF=\nabla\yy\,.
\end{align}
Applying the convective time derivative, after some algebraic manipulation
we obtain
\begin{align}\label{KLL+}
\DT{\Fe}=(\nabla\vv)\Fe-\Fe\Lp\ \ \ \text{ with }\ \ \Lp=\DT{\Fp}\Fpp^{-1}\,.
\end{align}
Furthermore, we apply the decomposition of $\Fe=J^{1/3}\Fedev$ to select
the isochoric part of the elastic distortion $\Fedev$. Using the calculus
$\DT{\Fe}=J^{1/3}\DT\Fedev+\frac13J^{-2/3}\DT J\Fedev
=J^{1/3}\DT\Fedev+\frac13J^{-2/3}({\rm div}\,\vv)\Fedev$, the kinematic
flow rule \eq{KLL+} turns into
\begin{align}\label{KLL++}
\DT{\Fedev}=\DEV(\nabla\vv)\Fedev-\Fedev\Lp\ \ \ \text{ with }\ \ \Lp=\DT{\Fp}\Fpp^{-1}\,.
\end{align}

Beside the kinematics, the further ingredient of the model is the (Helmholtz)
{\it free energy} $\psi$ depending on the elastic distortion $\Fe$ and temperature
$\theta$. This free energy $\psi=\psi(\FF,\theta)$ is considered per the
{\it actual volume} in the physical unit J/m$^3$ (not in J/kg for energy alternatively
considered in the referential frame). Possibly, $\psi$ may also depend on the phase
fraction $\chi$ as another intensive variable. The concept of stresses governed by
some potential (i.e.\ energy) is called {\it hyperelasticity}.
The leads to the conservative part of the
(actual)  {\it Cauchy stress} and the  {\it Mandel stress},
the  {\it entropy}, and the {\it heat capacity} respectively as
\begin{align}\nonumber
&\TT=\psi_{\Fe}'\!(\Fe,\theta)\Fe^\top\!\!+\psi(\Fe,\theta)\bbI\,,\ \ \ \ 
\MM=\Fe^\top\psi_{\Fe}'\!(\Fe,\theta)\,,
\\&
\eta=-\psi_\theta'(\Fe,\theta)\,,\ \text{ and }\ 
c(\Fe,\theta)=-\theta\psi_{\theta\theta}''(\Fe,\theta)
\label{stress-entropy}\end{align}
with $\bbI$ denoting the unit matrix.
In isotropic media, it is legitimate to confine on a special case when
the thermal coupling is only through the volumetric part by 
considering a special ansatz 
\begin{align}\label{special-thermo-case}
  \psi=\psi(\Fe,\theta)=
  \varphi(\Fedev)+\COUPLING(J,\theta)\ \ \text{ with }\ 
\Fedev=J^{-1/3}\Fe\ \text{ and }\  J=\det\Fe\,;
\end{align}
here $\varphi:{\rm SL}_3\to\R^+$ is assumed frame indifferent with
${\rm SL}_3:=\{F^*\in\R^{3\times3};\ \det F^*=1\}$ denoting the special linear
group. Then we obtain the stresses $\TT=\boldT(J,\Fedev,\theta)$ and
$\MM=\boldM(J,\Fedev,\theta)$ in the special forms 
\begin{subequations}\label{T-M-stress}\begin{align}\nonumber
 \boldT(J,\Fedev,\theta)&=\Big(\!J^{-1/3}\varphi'(\Fedev)
  {-}\tfrac13J^{-4/3}\varphi'(\Fedev)\Fe{\rm Cof}\Fe
  {+}\COUPLING_J'(J,\theta){\rm Cof}\Fedev\Big)\Fedev^\top\!\!
  +\big(\varphi(\Fedev){+}\COUPLING(J,\theta)\big)\bbI
  \\[-.2em]&=\DEV\big(\varphi'(\Fedev)\Fedev^\top\big)
  +\big(\varphi(\Fedev)+\COUPLING(J,\theta)+J\COUPLING_J'(J,\theta)\big)\bbI
  \ \ \ \text{ and}
\label{T-stress}\\\nonumber
\boldM(J,\Fedev,\theta)&=\Fe^\top\Big(J^{-1/3}\varphi'(\Fedev)
  -\tfrac13J^{-4/3}\varphi'(\Fedev)\Fe{\rm Cof}\Fe
  +\COUPLING_J'(J,\theta){\rm Cof}\Fe\Big)
 \\[-.2em]&=\DEV\big(\Fedev^\top\varphi'(\Fedev)\big)+J\COUPLING_J'(J,\theta)\bbI\,,
\label{M-stress}\end{align}\end{subequations}
where we used $(\det F)'={\rm Cof}\,F=JF^{-\top}$, i.e.\
the Jacobi formula and Cramer’s rule. Note that the deviatioric part
$\DEV\big(\Fedev^\top\!\varphi'(\Fedev)\big)$ of
$\boldM(J,\Fedev,\theta)$ is, in fact, temperature-independent.

The 3rd law of thermodynamics needs $\COUPLING_{J\theta}''(J,0)=0$, which
ensures that, at the absolute temperature zero,
the entropy $\eta=-\COUPLING_\theta'(J,0)$ is independent of $J$.

The further modelling issue is the choice of a particular rheological
viscoelastic model. We consider the Jeffreys (also called anti-Zener)
model, which standardly consists in the parallel Stokes model (with the
viscosity $\bbD_0$) with the Maxwell model composed from the (nonlinear)
elastic element governed by $\psi(\cdot,\theta)$ and another (here
in general nonlinear and temperature dependent) viscous element governed by
the viscoplastic potential $\zetap(\theta,\cdot)$ in series. 
In case of ice (glaciology), the conventional model is the 
the phenomenological polynomially-dependent flow rule
$M|\Lp|^{1/n-1}\Lp=\,$the driving stress, where $\Lp=\DT\Fp\Fpp^{-1}$ is
the inelastic distortion rate. For $n=1$, it would lead to a linear
creep (i.e.\ Newtonian fluid) but, for ice flow, the literature typically
considers {\it Glen}'s {\it flow} rule $n\sim3$
(leading to the specific non-Newtonian shear-thinning fluid)
and with $M=M(\theta)$. Here, we modify this
standard model by another viscosity $\bbD_1$ in parallel to the elastic
element, cf.\ Fig.~\ref{fig-rheology-Jeff-St-thermo}-left.
The driving stress  here is $\DEV\MM+\Feedev^\top\DD_1\Feedev^{-\top}$.
\begin{figure}[h]
\begin{center}
\psfrag{c}{\small$\bbC$}
\psfrag{d1}{\small$\bbD_0$}
\psfrag{C}{\small$\bbC$}
\psfrag{D1}{\small$\bbD_1$}
\psfrag{D2}{\small$\bbD_2$}
\psfrag{s}{\small$\sigma$}
\psfrag{Fe}{\small$\!\!\Fedev$}
\psfrag{Fp}{\small$\Fp$}
\psfrag{r}{\small$\varrho$}
\psfrag{C}{$\varphi'$}
\psfrag{y}{\small$\nabla\yy$}
\psfrag{D2}{$\zetap'$}
\psfrag{s}{\small${}$}
\psfrag{S0}{\small$\DEV\DD_0$}
\psfrag{S1}{\small$\!\!\!\DD_1$}
\psfrag{S2}{\small$\DEV\TT$}
\psfrag{M}{\small$\DEV\MM$}
\psfrag{q}{$\theta$}
\psfrag{q<qpt}{$\chi\sim0$}
\psfrag{q>qpt}{$\chi=0$}
\psfrag{Fe->I}{\small$\Fedev\!\stackrel{\sim}{\in}{\rm SO}_3$}
\psfrag{Fp->y}{\small$\!\Fp
  \!\!\stackrel{\sim}{\in}\!\!\frac{\nabla\yy{\rm SO}_3}{\det(\nabla\yy)^{1/3}}$}
\psfrag{S1->0}{\small$\!\!\!\!\DD_1{\sim}\,\bm0$}
\psfrag{S2->0}{\small$\DEV\TT{\sim}\,\bm0$}
\psfrag{melting}{\scriptsize\bf MELTING}
\psfrag{freezing}{\scriptsize\bf FREEZING}
\includegraphics[width=0.96\textwidth]{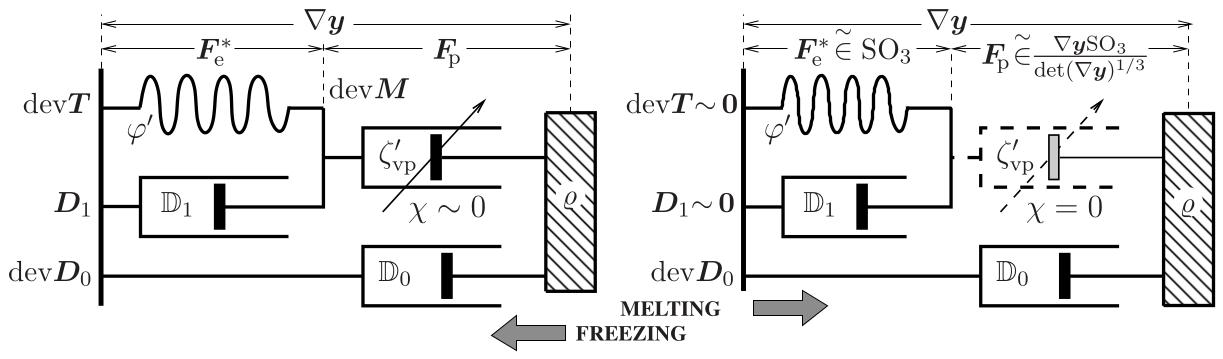}
\end{center}
\vspace*{-.6em}
\caption{{\sl A schematic 1-dimensional illustration the (nonlinear)
  Jeffreys-Stokesian rheological model in the deviatoric part
  involving the 
  phase-dependent
viscoplastic potential $\zetap(\chi,\cdot)$ in parallel with the Stokes
model involving the viscosity tensor $\bbD_0$, based on the multiplicative
decomposition \eq{KLL}. For complete melting 
for $\chi=1$, the model degenerates to Stokesian fluid
by putting $\zetap\equiv0$ which asymptotically (in the steady state)
causes $\DEV\TT=0$ and $\DD_1\equiv\bm0$ and, if $\varphi(\bbI)=0$,
also $\Fedev\in{\rm SO}_3:=\{F\in\R^{3\times3};\ F^\top F=\bbI\
\&\ \det F=1\}$.
}}
\label{fig-rheology-Jeff-St-thermo}
\end{figure}

In the linear temperature-independent case, this 4-element rheology
would be equivalent to the standard 3-element Jeffreys rheology. Yet,
it is not true if $\zetap(\theta,\cdot)$ is temperature dependent. This
temperature dependency is essential for the transition between
the Jeffreys rheology (as used for some solid-type material allowing
creep or plasticity in the deviatoric part) and the Stokes model for
fluids. Allowing  $\zetap(\theta,\cdot)$ to degenerate to zero
for $\theta>\theta_\text{\sc pt}$, we model the {\it complete melting}.
Then the model reduces to the Newtonian fluid. The role of $\bbD_1$ is
important for controlling the 
rate of
$\Fedev$, i.e.\ the deviatoric part of the elastic distortion rate
$\DT\Fe\Fee^{-1}$, denoted by $\Ld$, i.e.\ $\Ld=\DT\Fedev\Feedev^{-1}$,
even during the melting when $\zetap(\theta,\cdot)$
switches from a regular viscoplasticity to a complete degeneracy.
\begin{figure}[h]
\begin{center}
\psfrag{q}{\small $\theta$}
\psfrag{qsf}{\small $\theta_\text{\sc pt}$}
\psfrag{Gviscous}{\small $\GM=\GM(\theta)$}
\psfrag{T}{\small $\Ent\in\ENT(\theta)$}
\psfrag{TS}{\small $\Ent_\text{\sc s}$}
\psfrag{TF}{\small $\Ent_\text{\sc l}$}
\psfrag{solid elastic}{\!\footnotesize\sf\begin{minipage}[t]{10em}\hspace*{0em}solid with\\[-.3em]\hspace*{0em}dominating\\[-.3em]\hspace*{0em}elasticity (and\\[-.3em]\hspace*{0em}possible fracture
  \\[-.3em]\hspace*{0em}as in Remark~\ref{rem-fracture})
\end{minipage}}
\psfrag{creep}{\!\footnotesize\sf\begin{minipage}[t]{10em}\hspace*{0em}creep
  in solid ice\\[-.3em]\hspace*{0em}at temperatures\\[-.3em]\hspace*{0em}close to melting
\end{minipage}}
\psfrag{PT}{\!\footnotesize\sf\begin{minipage}[t]{10em}\hspace*{-1em}solid-liquid
  phase\\[-.3em]\hspace*{0em}transformation\\[-.3em]\hspace*{-.7em}(melting/solidification)
\end{minipage}}
\psfrag{temperature}{\footnotesize temperature}
\psfrag{melting}{\footnotesize melting/freezing} 
\psfrag{latent}{\footnotesize\ latent}
\psfrag{heat}{\footnotesize\sf heat $L$}
\psfrag{fluid}{\footnotesize\sf fluid (liquid)}
\psfrag{solid}{\footnotesize\sf\begin{minipage}[t]{10em}frozen phase\\[-.3em]\hspace*{2em}(ice)\end{minipage}}
\psfrag{liquid}{\footnotesize\sf\begin{minipage}[t]{10em}melted phase\\[-.3em]\hspace*{1.7em}(water)\end{minipage}}
\hspace*{-.5em}{\includegraphics[width=36em]{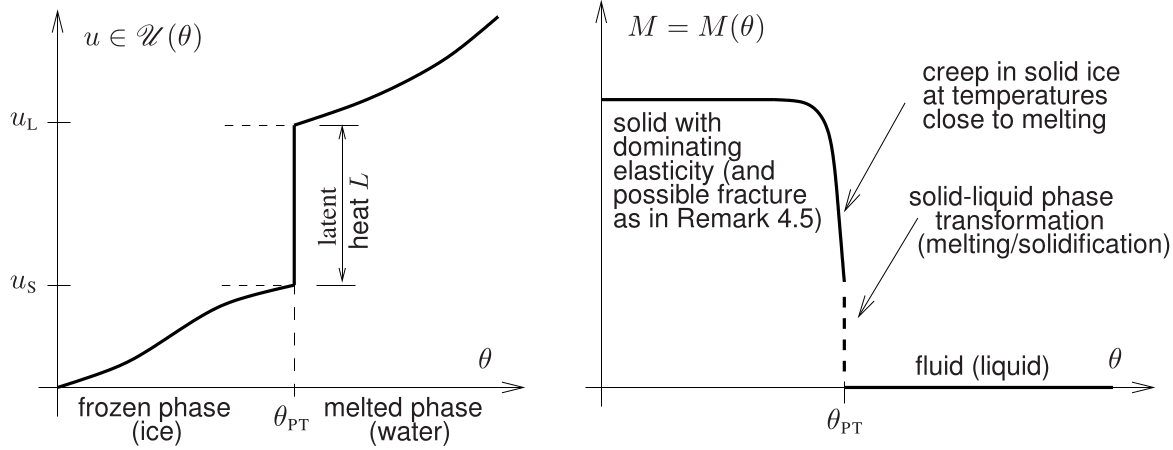}}
\end{center}
\vspace*{-1.5em}
\caption{{\sl
    A basic philosophy of the Stefan problem combined with
mechanical solid-liquid transition: temperature dependence of
    the thermal energy $\Ent=\ENT(\theta)$ in the classical
    Stefan problem (left)
    and the dependence of the 
    creep viscosity modulus $\GM$
in the non-quadratic potential
$\zetap(\theta,\cdot)=\GM(\theta)|\cdot|^{1+1/3}$ modelling a phenomenological
Glen's law for ice flow (right). As $\ENT$ is
    multivalued at the solid-liquid phase-transition temperature
    $\theta=\theta_\text{\sc pt}$ and thus the resulted heat capacity
 $\ENT'(\cdot)$ contains a Dirac measure of the magnitude $L$ 
    supported at $\theta=\theta_\text{\sc pt}$.
}}
\label{fig-Stefan}\end{figure}

In the volumetric part, there should not be any creep, so that we will use
the standard Kelvin-Voigt model. This split to the deviatoric and the volumetric
part is achieved simply by considering $\det\Fp=1$, which is granted
by ${\rm tr}\,\Lp=0$ together with the isochoric initial condition for $\Fp$.
Notably, $J=\det\Fe=\det\FF$ and $\varrho=\rhoR/J$ provided \eq{cont-eq+}
holds with the initial condition $\varrho|_{t=0}=\rhoR/J|_{t=0}$ with
$\rhoR$ a prescribed referential mass density. Thus we have $J=\rhoR/J$.

Reminding the ansatz \eq{special-thermo-case}, the internal energy
$\ENG=\ENG(J,\Fedev,\theta)$ determined by the Gibbs relation is now
additively coupled as
\begin{subequations}\label{def-of-energy}\begin{align}\nonumber
&\ENG(J,\Fedev,\theta)=\psi(J,\Fedev,\theta)-\theta\psi_\theta'(J,\Fedev,\theta)
    =\varphi(\Fedev)
    +\ENT(J,\theta)\
\\[-.3em]&\hspace{12.7em}
    \text{ with }\
\ENT(J,\theta)=\COUPLING(J,\theta)-\theta\COUPLING_\theta'(J,\theta)-\COUPLING(J,0)\,.
\intertext{Moreover, the ``thermal stress'' $\boldT(J,\Fedev,\theta)-\boldT(J,\Fedev,0)$
  reduces to a ``thermal pressure'' $\ADI$ as}
\nonumber
&\boldT(J,\Fedev,\theta)=\boldT(J,\Fedev,0)+\ADI(J,\theta)\,\bbI\ \
\\[-.5em]&\hspace{6em}
\text{ with }\
\ADI(J,\theta)=\COUPLING(J,\theta)+J\COUPLING_J'(J,\theta)-J\COUPLING_J'(J,0)-\COUPLING(J,0)\,;
\label{A-stress}
\end{align}\end{subequations}
note that $\ADI(J,\theta)$ obviously vanishes for $\theta=0$.
\begin{figure}[h]
\begin{center}
\psfrag{q}{\small $\theta$}
\psfrag{qsf}{\small $\theta_\text{\sc pt}$}
\psfrag{Gviscous}{\small $\GM=\GM(\theta{+}\chi)$}
\psfrag{TS}{\small $\!\theta_\text{\sc pt}\ $}
\psfrag{TF}{\small $\theta_\text{\sc pt}{+}1$}
\psfrag{solid}{\footnotesize\sf solid}
\psfrag{temperature}{\scriptsize\sf temperature}
\psfrag{melting}{\scriptsize\sf melting\,/\,freezing} 
\psfrag{qsf}{\small $\theta_\text{\sc pt}$}
\psfrag{liquid}{\footnotesize\sf liquid} 
\psfrag{mushy zone}{\footnotesize\sf\begin{minipage}[t]{7em}\hspace*{0em}mushy\\[-.3em]\hspace*{0.4em}zone\end{minipage}}
\psfrag{latent}{\scriptsize\sf latent}
\psfrag{heat}{\scriptsize\sf heat $\varrho\ell$}
\psfrag{fluid}{\scriptsize\sf fluid}
\psfrag{w}{$\theta{+}\chi$}
\hspace*{.5em}{\includegraphics[width=22em]{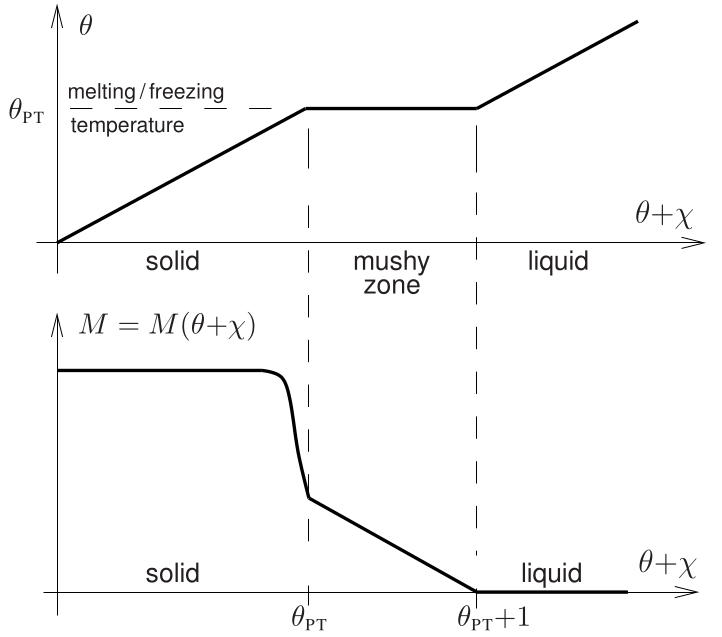}}
\end{center}
\vspace*{-1.5em}
\caption{{\sl
    The illustration of the ``continualization'' of the
    discontinuous viscous creep modulus $\GM$ as
    in Fig.\,\ref{fig-Stefan}-right by introducing the
    ``enhanced'' temperature $\theta{+}\chi$ as in
    Remark~\ref{rem-enhanced}.
}}
\label{fig-enhanced-temperature}\end{figure}

Again we adopt the natural ansatz \eq{special-thermo-case} now
merged with \eq{split-phi} considered for $L=\varrho\ell$: 
\begin{align}
\psi(J,\Fedev,\theta)=\varphi(\Fedev)+\wt\COUPLING(J,\theta)-\begin{cases}\qquad 0
&\text{if }\ \theta\le\theta_\text{\sc pt}\,,
\\[-.1em]\
\varrho\ell\,\big(\,\theta/\theta_\text{\sc pt}-1\big)
&\text{if }\ \theta>\theta_\text{\sc pt}\,.
    \end{cases}
\label{special-thermo-case+}\end{align}
Here $\ell$ is the latent heat in J/m$^3$.
Noteworthy, $\ell\chi$ itself with the physical dimension J/m$^3$ is an intensive variable
while $\varrho\ell\chi$ in J/kg is an extensive variable, so that we can sum
correctly two extensive variables $\varrho\ell\chi$ and $\wt\ENT(J,\theta)$ with 
$\wt\ENT(J,\theta)=\wt\COUPLING(J,\theta)-\theta\wt\COUPLING_\theta'(J,\theta)
-\wt\COUPLING(J,0)$. This gives the full (set-valued) {\it thermal} (part of the
internal) {\it energy}
\begin{align}
  \ENT(J,\theta)=\wt\ENT(J,\theta)+\varrho\ell\chi\ \ \ \text{ with }\ \
  \chi\in H(\theta{-}\theta_\text{\sc pt})\,.
\label{therma-engr}\end{align}
Note that 
$\ENT(J,0)=0$.

The specific entropy $\eta=-\COUPLING_\theta'(J,\theta)$ is an extensive variable
(in JK$^{-1}$m$^{-3}$) and its evolution-and-transport is governed by the entropy equation
\begin{align}\pdt\eta+{\rm div}\big(\vv\,\eta\big)
=\frac{\xi-{\rm div}\,{\bm j}}\theta\,,
\label{EUL-L-entropy-eq}\end{align}
where $\xi$ is the dissipation rate as specified below in 
\eq{EUL-L-Jeff-St6-thermo} and ${\bm j}=-\kappa(\theta,\chi)\nabla\theta$
is the heat flux governed by the Fourier law with the thermal
conductivity $\kappa=\kappa(\theta,\chi)$. Substituting
$\eta$ from \eq{stress-entropy} into \eq{EUL-L-entropy-eq},
i.e.\ here $\eta=-\COUPLING_\theta'(J,\theta)$, it leads to
$-\COUPLING_{\theta\theta}''(J,\theta)\DT\theta
  =\COUPLING_{J\theta}''(J,\theta)\DT J
  +({\rm div}\,\vv)\COUPLING_\theta'(J,\theta)
  +(\xi+{\rm div}(\kappa(\theta,\chi)\nabla\theta))/\theta$, i.e.\
  the heat transfer equation
\begin{align}
\!\!\!c(J,\theta)\DT\theta-{\rm div}\big(\kappa(\theta,\chi)\nabla\theta\big)
  =\theta\COUPLING_{J\theta}''(J,\theta)\DT J
  +\theta\,\COUPLING_\theta'(J,\theta)\,{\rm div}\,\vv
  +\xi
  \,,\!
\end{align}
where $c(J,\theta)=-\theta\COUPLING_{\theta\theta}''(J,\theta)$,
cf.\ \eq{stress-entropy}. This leads to the equation for the
thermal internal energy $\Ent=\ENT(J,\theta)$ as
\begin{align}\nonumber
  \pdt u+{\rm div}(u\vv)=\DT u+({\rm div}\,\vv)\,u
  &=\xi
  -{\rm div}(\kappa(\theta,\chi)\nabla\theta)+
  \big(\boldT(J,\Fedev,\theta)-\boldT(J,\Fedev,0)\big)\Colon\strain(\vv)
\\[-.5em]&=\xi
  -{\rm div}(\kappa(\theta,\chi)\nabla\theta)+\ADI(J,\theta){\rm div}\,\vv\,.
\end{align}

Involving still the kinetic overheating/undercooling by relaxing 
the inclusion in \eq{therma-engr} as in \eq{relaxed}, we obtain
the following system for $(\varrho,\vv,\Lp,\Fedev,\theta,\chi)$:
\begin{subequations}\label{EUL-L-Jeff-St-thermo}
\begin{align}\label{EUL-L-Jeff-St0-thermo}
&\DT\varrho=-\,({\rm div}\,\vv)\,\varrho\,,
\\[-.4em]&
\!\!\varrho\DT\vv={\rm div}\big(
\boldT(\rhoR/\varrho,\Fedev,\theta){+}\DD_1{+}\DD_0\big)
+\varrho\GRAVITY\,,
\label{EUL-L-Jeff-St1-thermo}
\\[-.4em]&\label{EUL-L-Jeff-St2-thermo}
\big[\zetap(\theta,\chi,\cdot)\big]'(\Lp)={\rm dev}
\big(\Feedev^\top\varphi'(\Fedev)\big)+\Feedev^\top\DD_1\Feedev^{-\top}\,,
\\[-.4em]&\label{EUL-L-Jeff-St4-thermo}
\DT{\Fedev}=\DEV(\nabla\vv)\Fedev-\Fedev\Lp\,,
\\[-.4em]\nonumber
&\DT\Ent={\rm div}\big(\kappa(\theta,\chi)\nabla\theta\big)
+\big(\ADI(\rhoR/\varrho,\theta){-}\Ent\big)\,{\rm div}\,\vv
+\xi\big(\theta,\chi;\vv,\Ld,\Lp\big)\,,
\\[-.3em]&
\hspace{2em}\text{where }\ \Ent=
\wt\ENT(\rhoR/\varrho,\theta)+\varrho\ell\chi\
\ \text{with }\
\wt\ENT(J,\theta):=\wt\COUPLING(J,\theta)-\theta\,\wt\COUPLING_\theta'(J,\theta)
-\wt\COUPLING(J,0)\,,\!
\label{EUL-L-Jeff-St5-thermo}
\\[-.2em]&\NU\,\DT\chi+H^{-1}(\chi)\ni
\Upsilon(\theta{-}\theta_\text{\sc pt})\,
\label{EUL-L-Jeff-St3-thermo}
\intertext{with $\boldT$ from \eq{T-stress} and $\ADI$ from \eq{A-stress} and with $\GRAVITY$ a prescribed acceleration (gravity).
The dissipative contributions to the Cauchy stress in
\eq{EUL-L-Jeff-St1-thermo} are}
&\DD_0=\bbD_0^{}\strain(\vv)\ \ \text{ and }\ 
\ \DD_1=\bbD_1^{}\Ld\ \ \text{ for }\ \ \Ld=\DT\Fedev\Feedev^{-1},
\intertext{and the corresponding dissipation rate
$\xi(\theta,\chi;\vv,\Ld,\Lp)$ in \eq{EUL-L-Jeff-St5-thermo} is}
&\xi\big(\theta,\chi;\vv,\Ld,\Lp\big)=
\bbD_0\strain(\vv)\Colon\strain(\vv)+\bbD_1\Ld\Colon\Ld+
\big[\zetap(\theta,\chi,\cdot)\big]'(\Lp)\Colon\Lp\,.
\label{EUL-L-Jeff-St6-thermo}
\end{align}\end{subequations}

Let us note that we now have naturally modified \eq{relaxed} by replacing the
partial time derivative by the convective time derivative in
\eq{EUL-L-Jeff-St3-thermo}, which is related with the attribute of
the phase-fraction variable $\chi$ as an intensive variable taking values
in $[0,1]$.

The system \eq{EUL-L-Jeff-St-thermo} is to be comleted by 
the boundary conditions the momentum and the heat equations.
As a usual simplification in the Eulerian frame, a fixed domain
(denoted by $\varOmega$) is to be considered with the
zero normal velocity prescribed. Thus, one can consider
\begin{align}
\nn\Cdot\vv=0\,,\ \ \ \ \ 
\big[(\boldT(\rhoR/\varrho,\Fedev,\theta){+}\DD_1{+}\DD_0)\nn\big]_\text{\sc t}^{}=\bm 0\,,\ \ 
\ \text{ and }\ \ \ \nn\Cdot\nabla\theta=0\,,
\label{basic-BC}\end{align}
where $[\cdot]_\text{\sc t}$ is the tangential component of the traction
vector, i.e.\ $[\tt]_\text{\sc t}=\tt-(\tt\Cdot\nn)\nn$.

The form of the viscosity contribution to the Mandel stress $\MM$ in
\eq{EUL-L-Jeff-St2-thermo} is motivated by the 
energetics behind \eq{EUL-L-Jeff-St-thermo} which can be seen by testing
\eq{EUL-L-Jeff-St1-thermo} by $\vv$ and \eq{EUL-L-Jeff-St2-thermo}
by $\Lp$.
Here, abbreviating $\TT=\boldT(J,\Fedev,0)=\DEV(\varphi'(\Fedev)\Feedev^\top)
+\varphi(\Fedev)\bbI$, we use
\begin{align}\nonumber
\!\!\int_\varOmega&\!{\rm div}\TT\Cdot\vv\,\d\xx
=\!\int_\varGamma\!\vv\Cdot\TT\nn\,\d S
-\!\!\int_\varOmega\!\DEV\big(\varphi'(\Fedev)\Feedev^\top\big)\Colon\nabla\vv
+\varphi(\Fedev){\rm div}\,\vv\,\d\xx
\\&\nonumber
=\!\int_\varGamma\!\vv\Cdot\TT\nn\,\d S
-\!\!\int_\varOmega\!\varphi'(\Fedev)\Colon\big(\DEV(\nabla\vv)\Fedev\big)+\varphi(\Fedev){\rm div}\,\vv\,\d\xx
\\&\nonumber
=\!\int_\varGamma\!\vv\Cdot\TT\nn\,\d S-\!\!\int_\varOmega\!
\varphi'(\Fedev)\Colon\big(\DT\Fedev\!{+}\Fedev\Lp\big)+\varphi(\Fedev){\rm div}\,\vv\,\d\xx
\\&\nonumber
=\!\int_\varGamma\!\vv\Cdot\TT\nn\,\d S-\!\!\int_\varOmega\!\varphi'(\Fedev)\Colon\DT\Fedev
+\Feedev^\top\varphi'(\Fedev)\Colon\Lp+\varphi(\Fedev){\rm div}\,\vv\,\d\xx
\\&\nonumber
=\!\int_\varGamma\!\vv\Cdot\TT\nn\,\d S-\!\int_\varOmega\!
\varphi'(\Fedev)\Colon\pdt{\Fedev}+\Feedev^\top\varphi'(\Fedev)\Colon\Lp\!
+\varphi'(\Fedev)\Colon(\vv\Cdot\nabla)\Fedev\!
  +\varphi(\Fedev){\rm div}\,\vv
  \,\d\xx
\\[-.0em]&
=\!\int_\varGamma\!\vv\Cdot\TT\nn\,\d S-\frac{\d}{\d t}\!\int_\varOmega\!
\varphi(\Fedev)\,\d\xx
-\!\int_\varOmega\!\big[\zetap(\theta,\cdot)\big]'(\Lp)\Colon\Lp\,\d\xx\,.
\label{Euler-hypoplast-test-momentum}\end{align}
Moreover, for merging the resulted terms
$(\bbD_1\DT{\Fedev}\Feedev^{-1})\Colon\nabla\vv
$ and ${\rm dev}(\Feedev^\top(\bbD_1\Ld)\Feedev^{-\top})\Colon\Lp
$, we use the calculus
\begin{align}\nonumber
\DD_1\Colon\nabla\vv-{\rm dev}&(\Fetop(\DD_1
\Fee^{-\top}))\Colon\Lp
\!\!\!\stackrel{\eq{EUL-L-Jeff-St1-thermo}}{=}\!
(\bbD_1\DT{\Fe}\Fee^{-1})\Colon\nabla\vv-(\Fetop(\bbD_1\DT{\Fe}\Fee^{-1}\Fee^{-\top}))\Colon\Lp
\\\nonumber
&
=(\bbD_1\Ld\Feedev^{-\top})\Colon(\nabla\vv)\Fedev
-(\bbD_1\Ld\Feedev^{-\top})\Colon(\Fedev\Lp)
\\&
\!\!\stackrel{\eq{EUL-L-Jeff-St3-thermo}}{=}\!(\bbD_1\Ld\Feedev^{-\top})\Colon\DT{\Fedev}
=
(\bbD_1\Ld)\Colon(\DT{\Fedev}\Feedev^{-1})=(\bbD_1\Ld)\Colon\Ld\,.
\label{EUL-L-Jeff-St-calcul}\end{align}
Here we used also the algebra $A\Colon (BC)=(B^\top A)\Colon C=(AC^\top)\Colon B$
for any three square matrices $A$, $B$, $C$.

Thus we obtain the expected energy-dissipation balance
\begin{align}
\!\frac{\d}{\d t}\int_\varOmega\frac\varrho2|\vv|^2\!
  +\varphi(\Fedev)+\wt\COUPLING(J,0)\,\d\xx+\!\int_\varOmega\!\xi(\chi;\vv,\Ld,\Lp)\,\d\xx
=\!\int_\varOmega\!\varrho\GRAVITY\Cdot\vv-\ADI(J,\theta)\,{\rm div}\,\vv\,\d\xx\!
\label{EUL-L-energy-Jeff-St}
\end{align}
with $J=\rhoR/\varrho$.
Besides, adding \eq{EUL-L-Jeff-St5-thermo} tested by 1, 
we obtain the total energy balance
\begin{align}
\frac{\d}{\d t}\int_\varOmega\frac\varrho2|\vv|^2
+\!\!\!\!\!\lineunder{\varphi(\Fedev)+\wt\COUPLING(J,0)+\wt\ENT(
J,\theta)}{$=:\wt\ENG(J,\Fedev,\theta)$}\!\!\!\!\!+\varrho\ell\chi\,\d\xx
=\int_\varOmega\varrho\GRAVITY\Cdot\vv\,\d\xx
\,.\label{EUL-L-energy-Jeff-St-thermo}
\end{align}
It should be noted that the relaxed evolution equation
\eq{EUL-L-Jeff-St3-thermo} determines phenomenologically the
overheating/undercooling effects completely out of the energetics.

\begin{remark}[An alternative form of the ``Stefan'' term in
\eq{special-thermo-case+}.]\upshape
The last term in \eq{special-thermo-case+} can equivalently use
$\rhoR/J$ instead of $\varrho$ with $J=\det\Fe$ and $\rhoR$ a given
``referential'' mass density. This variant, leading to write
\eq{special-thermo-case} with $\COUPLING(J,\theta)=\wt\COUPLING(J,\theta)
-\rhoR\ell\chi(\theta/\theta_\text{\sc pt}{-}1)^+/J$ with $\wt\COUPLING$ as
in \eq{special-thermo-case+}, is
conceptually more consistent since $\varrho$ does not explicitly
occur in the free energy and clearly reveals that this term
indeed does not influence the stress. Indeed, realizing that
$(1/\!\det\Fe)'=-(1/\!\det\Fe^2)\det'\Fe=-(1/\!\det\Fe^2){\rm Cof}\Fe$
so that $(1/\!\det\Fe)'\Fe^\top=-\bbI/\!\det\Fe$ due to
of the Jacobi formula 
${\rm Cof}\,A=\det'(A)$ and
the Cramer rule 
$A^{-1}\!={{\rm Cof}\,A^\top}\!/{\det A}$,
considering now
$\tilde\psi(\Fe,\theta)=\rhoR\ell(\theta/\theta_\text{\sc pt}{-}1)/\!\det\Fe$,
the corresponding contribution to the Cauchy stress
$$
\tilde\psi_{\Fe}'(\Fe,\theta)\Fe^\top+\tilde\psi(\Fe,\theta)\bbI
=\rhoR\ell\Big(1{-}\frac\theta{\theta_\text{\sc pt}}\Big)
\frac{({\rm Cof}\Fe)\Fe^\top}{\det\Fe^2}
+\rhoR\ell\Big(\frac\theta{\theta_\text{\sc pt}}{-}1\Big)\frac{1}{\det\Fe}\bbI=0
$$
and similarly also the contribution to the Mandel stress is zero.
\end{remark}

\section{The analysis of the model in the multipolar variant}\label{sec-anal}

We will prove existence of (suitably defined) weak solutions rather constructively
by a suitable time (semi)discretization combined with a certain truncation.
This will devise a computationally implementable conceptual strategy (up to
Remark~\ref{rem-full-discret} below), justified by a numerical stability and 
convergence of the approximate solutions thus obtained.
For simplicity, we restrict the model on $\zetap=\zetap(\chi,\cdot)$
and $\kappa=\kappa(\chi)$, although the more general case
$\zetap=\zetap(\theta,\chi,\cdot)$ and $\kappa=\kappa(\theta,\chi)$ in
Section~\ref{sec-thermo} would be analytically a simple extension.

However, it is generally recognized that any rigorous analysis of the
geometrically (and also materially) nonlinear models devised in this
chapter is problematic. Avoiding some very weak solution concepts
and correspondingly weak results and to facilitate a rigorous analysis using
the conventional weak-solution concepts, it is seems inevitable to involve
some gradient theories into the model. Here, when using the model in
the Eulerian frame, it is common to augment the dissipative part rather than
the conservative part (as common in the referential Lagrangian frame).

Here we apply a higher-order viscosity to the dampers 
$\bbD_0$ and $\bbD_1$ while the viscoplastic damper
$\zetap$ in Figure~\ref{fig-rheology-Jeff-St-thermo}, which
is assumed to vanish for $\theta>\theta_\text{\sc pt}$, remains
unchanged not to corrupt the possibility of complete degeneracy
due to complete melting. More specifically, we modify the model
\eq{EUL-L-Jeff-St-thermo} by involving the gradient theory
for dissipative stresses 
both $\DD_0$ and $\DD_1$. In the former case, we expand
$\DD_0=\bbD_0\strain(\vv)$ by the term
$-{\rm div}(\HYPER_0|\nabla\strain(\vv)|^{p-2}\nabla\strain(\vv))$ for some
$p>3$. In the latter case, we expand $\DD_1=\bbD_1\Ld$ by the term
$-{\rm div}(\HYPER_1|\nabla\Ld|^{q-2}\nabla\Ld)$
for some $q>3$. As a result, this expands the system 
(\ref{EUL-L-Jeff-St-thermo}a-f)
considered now with $\xi(\chi;\vv,\Ld,\Lp)$ as
\begin{subequations}\label{EUL-L-Jeff-St6-thermo-hyper}\begin{align}
&\DD_0=\bbD_0^{}\strain(\vv)
 -{\rm div}\big(\HYPER_0|\nabla\strain(\vv)|^{p-2}\nabla\strain(\vv)\big)\,,
\\[-.3em]&\label{EUL-L-Jeff-St6-thermo-hyper-D1}
\DD_1=\bbD_1^{}\Ld-{\rm div}\big(\HYPER_1|\nabla\Ld|^{q-2}\nabla\Ld\big)
\ \ \text{ for }\ \ \Ld=\DT\Fedev\Feedev^{-1},\ \ \text{ and }
\\&\nonumber
\xi(\chi;\vv,\Ld,\Lp)=
\bbD_0\strain(\vv)\Colon\strain(\vv)+\bbD_1\Ld\Colon\Ld
\\[-.2em]&\hspace{6.3em}
+\big[\zetap(\chi,\cdot)\big]'(\Lp)\Colon\Lp
+\HYPER_0|\nabla\strain(\vv)|^p+\HYPER_1|\nabla\Ld|^q\,.
\label{EUL-L-Jeff-St6-thermo-hyper-xi}
\end{align}\end{subequations}

As a result, recalling \eq{special-thermo-case+}, we consider the
system \eq{EUL-L-Jeff-St-thermo}
in terms of the momentum $\pp=\varrho\vv$ 
by using \eq{inertial} as:
\begin{subequations}\label{EUL-L-Jeff-St-thermo-grad}
\begin{align}\label{EUL-L-Jeff-St0-thermo-grad}
&\!\!\pdt\varrho=-\,{\rm div}\,\pp\ \ \text{ with }\ \pp=\varrho\vv
\,,
\\[-.0em]&
\!\!
\pdt\pp=\nabla\ADI(\rhoR/\varrho,\theta)+{\rm div}\Big(
\boldT(\rhoR/\varrho,\Fedev,0){+}\DD_1{+}\DD_0
{+}\pp{\otimes}\vv\Big)+\varrho\GRAVITY\,,
\label{EUL-L-Jeff-St1-thermo-grad}
\\[-.4em]&\label{EUL-L-Jeff-St2-thermo-grad}
\big[\zetap(\chi,\cdot)\big]'(\Lp)={\rm dev}
\big(\Feedev^\top\varphi'(\Fedev)\big)+\Feedev^\top\DD_1\Feedev^{-\top}\,,
\\[-.1em]&\label{EUL-L-Jeff-St4-thermo-grad}
\pdt{\Fedev}=\Ld\Fedev-(\vv\Cdot\nabla)\Fedev\ \ \ \text{ with }\ 
\Ld={\rm dev}(\nabla\vv)-\Fedev\Lp\Feedev^{-1},
\\[-.1em]\nonumber
&\pdt\Ent={\rm div}\Big(\kappa(\chi)\nabla\theta
-\Ent\vv\Big)
+\ADI(\rhoR/\varrho,\theta)\,{\rm div}\,\vv
+\xi\big(\chi;\vv,\Ld,\Lp\big)\,,
\\[-.5em]&
\hspace{4em}\text{where }\ \Ent=
\wt\ENT(\det\Fe,\theta)+\varrho\ell\chi\
\ \text{with }\
\wt\ENT(J,\theta):=\wt\COUPLING(J,\theta)-\theta\,\wt\COUPLING_\theta'(J,\theta)\,,
\label{EUL-L-Jeff-St5-thermo-grad}
\\[-.2em]&\pdt\chi+H^{-1}(\chi)\ni
\NU^{-1}\Upsilon
(\theta{-}\theta_\text{\sc pt})-\vv\Cdot\nabla\chi\,
\label{EUL-L-Jeff-St3-thermo-grad}
\end{align}\end{subequations}
with 
$\DD_0$, $\DD_1$, and $\xi(\chi;\vv,\Ld,\Lp)$ from
\eq{EUL-L-Jeff-St6-thermo-hyper} and with $\ADI$ from \eq{A-stress}.

We accompany the system \eq{EUL-L-Jeff-St-thermo-grad} with
\eq{EUL-L-Jeff-St6-thermo-hyper} by the boundary conditions enhancing
\eq{basic-BC} as
\begin{subequations}\label{enhanced-BC}\begin{align}
&\nn\Cdot\vv=0\,,\ \ \ \ \ 
\big[(
  \boldT(\rhoR/\varrho,\Fedev,\theta){+}\DD_1{+}\DD_0)\nn
+\divS(\mathfrak{H}\nn)\big]_\text{\sc t}^{}=\bm 0\,,\ \ \
\\&\mathfrak{H}\Colon(\nn{\otimes}\nn)=\bm0\,,\ \ \ \ 
(\nn\Cdot\nabla)\Ld=\bm0\,,\ \ 
\ \text{ and }\ \ \ \nn\Cdot\nabla\theta=0\,,
\end{align}\end{subequations}
where $\divS$ is the 2-dimensional surface divergence defined as
$\divS={\rm tr}(\nablaS)$ with
$\nablaS v=\nabla v-\frac{\partial v}{\partial\nn}\nn$
with ${\rm tr}(\cdot)$ denoting the trace of a
2${\times}$2-matrix and $\nablaS v$ denoting the surface gradient of $v$,
and where the so-called hyperstress is
$\mathfrak{H}=\HYPER_0|\nabla\strain(\vv)|^{p-2}\nabla\strain(\vv)$.
Moreover, we consider the initial-value problem with
the initial conditions for (\ref{EUL-L-Jeff-St-thermo-grad}a,b,d-f) with
\begin{align}
\varrho|_{t=0}=\rhoR/\!\det\Fezero\,,\ \ \ \ 
\Fedev|_{t=0}=\Fezero/\!\det{\Fezero}^{\!\!1/3},\ \ \ \ 
\theta|_{t=0}=\theta_0\,,\ \ \text{ and }\ \ \chi|_{t=0}=\chi_0\,.
\label{EUL-L-PT-IC}\end{align}

The energetics again leads to the energy-dissipation balance
\eq{EUL-L-energy-Jeff-St} now with $\xi$ from
\eq{EUL-L-Jeff-St6-thermo-hyper-xi}. This is based also on the
calculus \eq{EUL-L-Jeff-St-calcul} here enhanced as
\begin{align}\nonumber
&\DD_1\Colon\nabla\vv-{\rm dev}(\Feedev^\top(\DD_1
\Feedev^{-\top}))\Colon\Lp
=\Big(\bbD_1\Ld
-{\rm div}(\HYPER_1|\nabla\Ld|^{q-2}\nabla\Ld)
\Big)\Colon\nabla\vv
\\[-.4em]\nonumber
&\hspace{15.5em}-\Big(\Feedev^\top\Big(\bbD_1\Ld
-{\rm div}\big(\HYPER_1|\nabla\Ld|^{q-2}\nabla\Ld\big)\Big)\Feedev^{-\top}\Big)\Colon\Lp
\\[-.1em]\nonumber
&=
\Big(\bbD_1\Ld\Feedev^{-\top}\!\!-{\rm div}(\HYPER_1|\nabla\Ld|^{q-2}\nabla\Ld)\Feedev^{-\top}\Big)\Colon(\nabla\vv)\Fedev
\\[-.4em]\nonumber
&\hspace{15.5em}-\Big(\Big(\bbD_1\Ld
-{\rm div}(\HYPER_1|\nabla\Ld|^{q-2}\nabla\Ld)\Big)\Feedev^{-\top}\Big)\Colon(\Fedev\Lp)
\\[-.3em]&\nonumber
=
\Big(\bbD_1\Ld\Feedev^{-\top}\!\!-{\rm div}\big(\HYPER_1|\nabla\Ld|^{q-2}\nabla\Ld\big)
\Feedev^{-\top}\Big)\Colon\DT{\Fedev}\!
\\[-.1em]&=
\Big(\bbD_1\Ld\!-{\rm div}(\HYPER_1|\nabla\Ld|^{q-2}\nabla\Ld)\Big)\Colon(\DT{\Fedev}\Feedev^{-1})
=\Big(\bbD_1\Ld\!-{\rm div}\big(\HYPER_1|\nabla\Ld|^{q-2}\nabla\Ld)\Big)
\Colon\Ld.\!
\label{EUL-L-Jeff-St-calcul+}\end{align}
Moreover, adding \eq{EUL-L-Jeff-St5-thermo-grad} tested by 1, we
obtain the total energy balance \eq{EUL-L-energy-Jeff-St-thermo}.

We cast an elliptic equation for
$\Fedev$ which can serve for a weak formulation. To this aim,
we avoid the analytically problematic multiplication of $\Fetop$ and $\Fee^{-\top}$
in \eq{EUL-L-Jeff-St2-thermo-grad} by $\DD_1$ which is generally a
distribution due to the gradient term ${\rm div}(\HYPER_1|\nabla\Ld|^{q-2}\nabla\Ld)$.
Instead, we write \eq{EUL-L-Jeff-St2-thermo-grad} in the form
\begin{align}\nonumber
{\rm div}\big(\HYPER_1|\nabla\Ld|^{q-2}\nabla\Ld\big)-\bbD_1\Ld\!
+\Feedev^{-\top}\big[\zetap(\chi,\cdot)\big]'
\Big(\Feedev^{-1}\DEV(\nabla\vv{-}\Ld)\Fedev\Big)\Feedev^{\top}\ 
\\[-.2em]=\DEV\big(\varphi'(\Fedev)\Feedev^\top\big).\!
\label{EUL-L-PT-Le-equation}\end{align}
Moreover, we use the Green formula
$\int_\varOmega\vv{\cdot}\nabla\chi(\widetilde\chi{-}\chi)\,\d\xx=
\int_\varOmega\frac12({\rm div}\,\vv)\chi^2-\chi{\rm div}(\vv\widetilde\chi)\,\d\xx$
for the weak formulation (as an inequality) of the inclusion
\eq{EUL-L-Jeff-St3-thermo-grad}. 

\begin{definition}[Weak formulation of
\eq{EUL-L-Jeff-St6-thermo-hyper}--\eq{EUL-L-PT-IC}]\label{def-thermo-Ch5-PT}
The seven-tuple $(\varrho,\vv,\Lp,\Ld,\Fedev,\theta,\chi)$ with
$\varrho\in L^\infty(I{\times}\varOmega)\cap W^{1,1}(I{\times}\varOmega)$
with $\inf_{I\times\varOmega}\varrho>0$, $\vv\in 
L^p(I;W^{2,p}(\varOmega;\R^3))$,
$\Ld\in L^1(I;W^{1,q}(\varOmega;\Rtr))$,
$\Fedev\in 
W^{1,1}(I{\times}\varOmega;\R^{3\times3})$ with
$\det\Fedev=1$, 
 $\theta\in L^1(I;W^{1,1}(\varOmega))$, and
 $\chi\in L^\infty(I{\times}\varOmega)$ with $0\le\chi\le1$ a.e.,
such that $\ADI(\rhoR/\varrho,\theta)\in L^{p'}(I;L^1(\varOmega))$,
and $\ENT(\rhoR/\varrho,\theta)\in L^{p'}(I;L^1(\varOmega))$
is called a weak solution to
the system \eq{EUL-L-Jeff-St6-thermo-hyper} with the boundary
conditions \eq{enhanced-BC} and the initial conditions \eq{EUL-L-PT-IC} if
\begin{subequations}\label{def-thermo-Ch5-momentum}\begin{align}
&\nonumber
\!\int_0^T\!\!\!\int_\varOmega\bigg(\Big(\boldT(\frac{\rhoR}\varrho,\Fedev,0)
+\bbD_0\strain(\vv)+\bbD_1\Ld-\varrho\vv{\otimes}\vv\Big){:}\strain(\widetilde\vv)
-\varrho\vv{\cdot}\pdt{\widetilde\vv}
\\[-.3em]&\nonumber\hspace{1em}
+\ADI(\frac{\rhoR}\varrho,\theta)
\,{\rm div}\,\vv+\Big(\HYPER_0|\nabla\strain(\vv)|^{p-2}\nabla\strain(\vv)
+\HYPER_1|\nabla\Ld|^{q-2}\nabla\Ld\Big)\Vdots\nabla^2\widetilde\vv\bigg)\,\d\xx\d t
\\[-.6em]&\hspace{17.5em}
=\!\int_0^T\!\!\!\int_\varOmega\varrho\GRAVITY{\cdot}\widetilde\vv\,\d\xx\d t
+\!\int_\varOmega\!\varrho_0\vv_0{\cdot}\widetilde\vv(0)\,\d\xx
\label{def-thermo-Ch5-momentum1}
\intertext{holds for any $\widetilde\vv$ with $\widetilde\vv{\cdot}\nn={\bm0}$
on $I{\times}\varGamma$ and $\widetilde\vv(T)=0$, and with
$\Ld$ satisfying}\nonumber
&\!\int_\varOmega\!
\HYPER_1|\nabla\Ld|^{q-2}\nabla\Ld\Vdots\nabla\wt{\bm L}
+(\bbD_1\Ld{-}\DD_1)\Colon\wt{\bm L}\,\d\xx=0\!
\\[-.4em]&\hspace{2em}{\rm with }\ \DD_1=\Feedev^{-\top}\big[\zetap(\chi,\cdot)\big]'
\Big(\Feedev^{-1}\DEV(\nabla\vv{-}\Ld)\Fedev\Big)\Feedev^{\top}\!\!-\DEV\big(\varphi'(\Fedev)\Feedev^\top\big)
\intertext{for any $\wt{\bm L}\in W^{1,q}(\varOmega;\Rdev)$ with
  $(\nn\Cdot\nabla)\wt{\bm L}=\bm0$ and for a.a.\ $t\in I$, and}
\nonumber
&\!\!\!\!\int_0^T\!\!\!\int_\varOmega\!\Ent\pdt{\wt\theta}
+\Big(\ADI\big(\frac{\rhoR}\varrho,\theta\big)\,{\rm div}\,\vv
+\xi\big(\chi;\vv,\Ld,\Lp\big)\Big)\,\wt\theta+
\Big(\kappa(\chi)\nabla\theta{+}\Ent\vv\Big)\Cdot\nabla\wt\theta\,\d\xx\d t
\\[-.3em]&\hspace{4.5em}
+\!\!\int_\varOmega\!\!\big(\wt\ENT(J_0,\theta_0){+}\varrho_0\ell\chi_0\big)
\,\wt\theta(0)\,\d\xx
=0\ \ \ \text{ with }\ \ \Ent=
\wt\ENT\big(\frac{\rhoR}\varrho,\theta\big)+\varrho\ell\chi
\label{def-thermo-Ch5-momentum3}
\intertext{with 
$\xi(\chi;\vv,\Ld,\Lp)$ from \eq{EUL-L-Jeff-St6-thermo-hyper}
holds for any $\wt\theta$ smooth with $\wt\theta(T)=0$ and
$\nn{\cdot}\nabla\wt\theta=0$ on $I{\times}\varGamma$, and eventually}
&\nonumber
\int_0^T\!\!\!\int_\varOmega\!
\frac{\Upsilon(\theta{-}\theta_\text{\sc pt})}{\NU}
(\widetilde\chi-\chi)+\chi\pdt{\widetilde\chi}
-\Frac12({\rm div}\,\vv)\chi^2+\chi{\rm div}(\vv\widetilde\chi)\,\d\xx\d t
\\[-.5em]&\label{def-thermo-Ch4-chi-rule}\hspace*{8em}
+\int_\varOmega\Frac12\chi_0^2-\chi_0^{}\chi(0)\,\d\xx
\ge\int_\varOmega\!\Frac12\chi(T)^2-\chi(T)\widetilde\chi(T)\,\d\xx
\end{align}\end{subequations}
holds for any $\widetilde\chi\in W^{1,\infty}(I{\times}\varOmega)$ with
$0\le\widetilde\chi\le1$ a.e.\ on $I{\times}\varOmega$,
and if \eq{def-thermo-Ch4-chi-rule} holds for any
$\wt\chi\in W^{1,\infty}(I{\times}\varOmega)$ with $0\le\wt\chi\le1$ a.e., and if 
\eq{EUL-L-Jeff-St0-thermo-grad} and \eq{EUL-L-Jeff-St4-thermo-grad}
hold a.e.\ on $I{\times}\varOmega$ together with the respective initial conditions
for $\varrho$ and $\Fedev$ in \eq{EUL-L-PT-IC}.
\end{definition}

We devise the partially decoupled truncated time discretization of the system
\eq{EUL-L-Jeff-St6-thermo-hyper}--\eq{EUL-L-Jeff-St-thermo-grad}.
We use the equidistant partition of the time interval $I=[0,T]$ with the
time step $\tau>0$, assuming $T/\tau$ integer. We denote by $\varrho_\TAU^k$,
$\vv_\TAU^k$, $\FedevTAU^{\!\!k}$, $\theta_\TAU^k$, ...
the approximate values of $\varrho$, $\vv$, $\Fedev$, $\theta$, ...
at time instants $t=k\tau$ with $k=1,2,...,T/\tau$. 
We do not make a full time discretization but use the mapping
$\boldG:I{\times}\R^3{\times}\Rtr{\times}\R^{3\times3}\to\R^{3\times3}$
defined by $\GG(t)=\boldG(t,\vv,\LL,\GG_0)$
with $\GG=\GG(t)$ being the unique solution to the initial-value problem
\begin{align}
\pdt{\GG}=\LL\GG(t)-(\vv\Cdot\nabla)\GG(t)\
\ \ \text{ with }\ \ \GG(0)=\GG_0\,.
\end{align}
It is important that, since ${\rm tr}\LL=0$,
from $\det\GG_0=1$ we obtain $\det\GG(t)=1$ for all $t>0$.
In detail, using the Jacobi formula and the Cramer rule, we indeed have
\begin{align*}
  \pdt{}J+\vv\Cdot\nabla J&
  =({\rm Cof}\GG)\Colon\Big(\pdt{}\GG+(\vv\Cdot\nabla)\GG\Big)
=({\rm Cof}\GG)\Colon\big({\rm dev}(\nabla\vv)\GG-\GG\LL\big)
\\&=(({\rm Cof}\GG)\GG^\top)\Colon({\rm dev}(\nabla\vv)
-(\GG^\top{\rm Cof}\GG)\Colon\LL
=J\bbI\Colon({\rm dev}(\nabla\vv)-\LL)
=0
\end{align*}
with $J=\det\GG$, so that the
constant initial condition $\det\GG_0=1$ is transported as the constant = 1.

We use it for keeping the ``semi-discretization'' of the flow-rule for
\eq{EUL-L-Jeff-St4-thermo-grad} because, full time discretization
$\frac{\FedevTAU^{\!\!k}{-}\FedevTAU^{\!\!k-1}\!\!}\tau\,=\LdTAUk\FedevTAU^{\!\!k}
-(\vv_\TAU^k\Cdot\nabla)\FedevTAU^{\!\!k}$ could not keep
$\det\FedevTAU^{\!\!k}$ equal 1 (and even positive), so that 
$\FedevTAU^{\!\!k}$ could even not to be invertible while
replacing $(\FedevTAU^{\!\!k})^{-1}$ with $({\rm Cof}\FedevTAU^{\!\!k})^{\top}$
would bring some difficulties either, cf.\ also Remark~\ref{rem-full-discret} below
for further discussion.

Altogether, considering a sufficiently large threshold $\THRESHOLD$, we devise
the time semi-discretization of 
\eq{EUL-L-Jeff-St-thermo-grad} with (\ref{EUL-L-Jeff-St6-thermo-hyper}a,b) after
the truncation for $(\varrho_\TAU^k,\pp_\TAU^k,\LdTAUk,\LpTAU^k,\FedevTAU^{\!\!k},\theta_\TAU^k)$,
and thus also $\vv_\TAU^k$ and $\Ent_\TAU^k$, as 
\begin{subequations}\label{EUL-L-visco-thermo-PT+disc}
\begin{align}\label{EUL-L-visco-thermo-PT+0disc}
&\frac{\varrho_\TAU^k{-}\varrho_\TAU^{k-1}\!\!}\tau\,=
-\,{\rm div}\,\pp_\TAU^k
\ \ \ \text{ with }\ \,\pp_\TAU^k=\varrho_\TAU^k\vv_\TAU^k\,,\!
\\[-.2em]&\nonumber
\frac{\pp_\TAU^k{-}\pp_\TAU^{k-1}\!\!}\tau\,=\frac{
{\rm div}\boldT(J_\TAU^{k-1},\FedevTAU^{\!\!k-1},0)+
\nabla\ADI(
J_\TAU^{k-1},\theta_\TAU^{k-1})}{\!1+\big(\|\wt\Eng_\TAU^{k-1}\|_{L^1(\varOmega)}{-}\THRESHOLD\big)^+\!}
+{\rm div}\Big(\DD_{1,\TAU}^k+\DD_{0,\TAU}^k
-\pp_\TAU^k{\otimes}\vv_\TAU^k\Big)+\varrho_\TAU^k\GRAVITY_{\tau}^k
\\[-.0em]&\nonumber
\hspace{4.3em}\text{with }\ \DD_{1,\TAU}^k=
\bbD_1\LdTAUk-{\rm div}\big(\HYPER_1|\nabla\LdTAUk|^{q-2}\nabla\LdTAUk\big)
\,,\ \ J_\TAU^{k-1}=\rhoR/\varrho_\TAU^{k-1}\,,
\!\!\!\\[-.2em]&
\hspace{4.3em}\text{and }\ \ 
\DD_{0,\TAU}^k=\bbD_0\strain(\vv_\TAU^k)-{\rm div}\,\mathfrak{H}_\TAU^k
\ \ \text{ with }\ \ 
\mathfrak{H}_\TAU^k=
\HYPER_0\big|\nabla\strain(\vv_\TAU^k)\big|^{p-2}\nabla\strain(\vv_\TAU^k)\,,
\label{EUL-L-visco-thermo-PT+1disc}
\\[-.1em]&
\big[\zetap(\chi_\TAU^{k-1},\cdot)\big]'(\LpTAU^k)
=\frac{\DEV\boldM(\FedevTAU^{\!\!k-1})}{\!
  1+\big(\|\wt\Eng_\TAU^{k-1}
  \|_{L^1(\varOmega)}{-}\THRESHOLD\big)^+\!}
+(\FedevTAU^{\!\!k-1})^\top\DD_{1,\TAU}^k(\FedevTAU^{\!\!k-1})^{-\top},
\nonumber\\[-1.7em]\label{EUL-L-visco-thermo-PT+.disc}
\\[.1em]
&\FedevTAU^{\!\!k}=\boldG\big(\tau,\vv_\TAU^k,\LdTAUk,\FedevTAU^{\!\!k-1}\big)
\ \ \text{ with }\
\LdTAUk=\DEV(\nabla\vv_\TAU^k)-\FedevTAU^{\!\!k-1}\LpTAU^k(\FedevTAU^{\!\!k-1})^{-1}\,,
\label{EUL-L-visco-thermo-PT+2disc}
\\[.1em]&\nonumber
\frac{\Ent_\TAU^k{-}\Ent_\TAU^{k-1}\!\!}\tau\,
=\frac{\ADI(J_\TAU^k,\theta_\TAU^k)\,{\rm div}\,\vv_\TAU^k}{\!
  1+\big(\|\wt\Eng_\TAU^k\|_{L^1(\varOmega)}{-}\THRESHOLD\big)^+\!}
+{\rm div}\Big(\kappa(
\chi_\TAU^k)\nabla\theta_\TAU^k-\Ent_\TAU^k\vv_\TAU^k\Big)
+\xi(\chi_\TAU^k;\vv_\TAU^k,\LdTAUk,\LpTAU^k)
\\[-.7em]&\hspace{13.5em}
\text{with }\ \ \Ent_\TAU^k=
\wt\ENT(J_\TAU^k,\theta_\TAU^k)+\varrho_\TAU^k\ell\chi_\TAU^k
\,\ \ \text{ and}\ \ J_\TAU^k=\rhoR/\varrho_\TAU^k\,,
\label{EUL-L-visco-thermo-PT+3disc}
\\[-.5em]&
\frac{\chi_\TAU^k{-}\chi_\TAU^{k-1}\!\!}\tau\,+H^{-1}(\chi_\TAU^k)
\ni\NU^{-1}\Upsilon(\theta_\TAU^k{-}\theta_\text{\sc pt})-\vv_\TAU^k\Cdot\nabla\chi_\TAU^k
\label{EUL-L-visco-thermo-PT+4disc}
\end{align}\end{subequations}
with $\GRAVITY_{\tau}^k:=\int_{\tau(k-1)}^{\tau k}\GRAVITY(t)\,\d t$ in
\eq{EUL-L-visco-thermo-PT+1disc} and with $\wt\ENT$ from
\eq{EUL-L-Jeff-St5-thermo-grad}. Here we used the abbreviation 
$\wt\Eng_\TAU^{k-1}=\wt\ENG(
J_\TAU^{k-1},\FedevTAU^{\!\!k-1},\theta_\TAU^{k-1})$ and
$\wt\Eng_\TAU^k=\wt\ENG(J_\TAU^k,\FedevTAU^{\!\!k},\theta_\TAU^k)$.

This scheme is partly decoupled due to the terms with
$\theta_\TAU^{k-1}$ and $\chi_\TAU^{k-1}$ in \eq{EUL-L-visco-thermo-PT+1disc}
and \eq{EUL-L-visco-thermo-PT+.disc}. Specifically, the system
(\ref{EUL-L-visco-thermo-PT+disc}a--d) is to be solved for
$(\varrho_\TAU^k,\pp_\TAU^k,\LdTAUk,\LpTAU^k,\FedevTAU^{\!\!k})$
and thus also for $\vv_\TAU^k$ and, then,
(\ref{EUL-L-visco-thermo-PT+disc}e,f) is to be solved for
$(\theta_\TAU^k,\chi_\TAU^k)$.

We will use the notation that exploits the interpolants
defined as follows: using the values $(\vv_\tau^k)_{k=0}^{T/\tau}$, we define the
piecewise affine and piecewise constant (forward or backward) interpolants
respectively as
\begin{align}\nonumber\\[-2.2em]\nonumber
&\vv_\tau(t)\!:=\Big(\frac t\tau\,{-}\,k{+}1\Big)\,\vv_\tau^k
\!+\Big(k\,{-}\,\frac t\tau\Big)\,\vv_\tau^{k-1}\ \ 
\text{ and }\ \ 
\\[-.0em]&
\overline\vv_\tau(t)\!:=\vv_\tau^k
\,,\ \ \text{ and }\ \ \underline\vv_\tau(t)\!:=\vv_\tau^{k-1}
\ \ \text{ for }\ \ (k{-}1)\tau<t\le k\tau\ 
\label{def-of-interpolants-Ch3}\end{align}
with $k=0,1,...,T/\tau$. Analogously, we define also $\pp_\tau$,
$\overline\pp_\tau$, $\oFedevT$,  and $\uFedevT$, etc.
Written compactly in terms of these interpolants now also with
defining
$$
\tildeFedevTAU(t)
=\boldG\big(t,\vv_\TAU^k,\LpTAU^k,\FedevTAU^{\!\!k-1}\big)
\ \ \text{ on }\ t\in[(k{-}1\tau,k\tau]\ \ \text{ for }\ \
k=1,...,T/\tau\,,
$$
the system \eq{EUL-L-visco-thermo-PT+disc} looks as
\begin{subequations}\label{EUL-L-visco-thermo-PT+disc+}
\begin{align}\label{EUL-L-visco-thermo-PT+0disc+}
&\!\!\pdt{\varrho_\TAU}=-\,{\rm div}\,\opT
\ \ \ \text{ with }\ \,\opT=\orT\ovT\,,\!
\\[-.2em]&\nonumber
\!\!\pdt{\pp_\TAU}=
\frac{{\rm div}\,
\boldT(\uJT,\uFedevT,0)+
\nabla\ADI(
\uJT,\utT)}{\!  1{+}\big(\|\ueT\|_{L^1(\varOmega)}{-}\THRESHOLD\big)^+\!}
+{\rm div}\big(\oDTone
+\oDTzero
-\opT{\otimes}\ovT\big)+\orT\overline{\GRAVITY}_{\tau}
\\[-.1em]&\nonumber
\hspace{4em}\text{with }\ \oDTone=
\bbD_1\oLdT-{\rm div}\big(\HYPER_1|\nabla\oLdT|^{q-2}\nabla\oLdT\big)\,,
\ \ \ \uJT=\rhoR/\urT\,,
\\[-.2em]&
\hspace{4em}\text{and }\ \ \oDTzero=\bbD_0\strain(\ovT)
-{\rm div}\bar{\mathfrak{H}}_\TAU
\ \ \text{ with }\ \ \bar{\mathfrak{H}}_\TAU=
\HYPER_0\big|\nabla\strain(\ovT)\big|^{p-2}\nabla\strain(\ovT)
\,,
\label{EUL-L-visco-thermo-PT+1disc+}
\\[-.1em]&
\!\big[\zetap(\uxT,\cdot)\big]'(\oLpT)
=\frac{\DEV\boldM(\uFedevT)}{\!
  1{+}\big(\|\ueT
  \|_{L^1(\varOmega)}{-}\THRESHOLD\big)^+\!}
+\uFedevT^{\!\!\!\top}\oDTone\uFedevT^{\!\!\!-\top}
\,,\label{EUL-L-visco-thermo-PT+4disc+}
\\[-.4em]
&\!\!
\pdt{\tildeFedevTAU}
=\oLdT\tildeFedevTAU-(\ovT\Cdot\nabla)\tildeFedevTAU
\ \ \text{ with }\
\oLdT=\DEV(\nabla\ovT)-\uFedevT\oLpT\uFedevT^{\!\!\!-1}\,,
\!\!\!
\label{EUL-L-visco-thermo-PT+2disc+}
\\[-.0em]&\nonumber
\!\!\pdt{\Ent_\TAU}=\frac{\ADI(\oJT,\otT)\,{\rm div}\,\ovT}{\!
  1+\big(\|
  \oeT\|_{L^1(\varOmega)}{-}\THRESHOLD\big)^+\!}
+{\rm div}\big(\kappa(\oxT)\nabla\otT-\ouT\ovT\big)
+\xi(\oxT;\ovT,\oLdT,\oLpT)\ \ \ 
\\[-.7em]&
\hspace{11em}\text{with }\ \ \ouT
=\wt\ENT(\oJT,\otT)+\orT\ell\oxT
\ \ \text{ and}\ \oJT=\rhoR/\orT\,,
\label{EUL-L-visco-thermo-PT+3disc+}
\\[-.3em]&\!\!
\pdt{\chi_\TAU}+H^{-1}(\oxT)\ni
\NU^{-1}\Upsilon(\otT{-}\theta_\text{\sc pt})-\ovT\Cdot\nabla\oxT\,.
\label{EUL-L-visco-thermo-PT+5disc+}
\end{align}\end{subequations}
where $\ueT=\wt\ENG(\uJT,\uFedevT,\utT)$ and
$\oeT=\wt\ENG(\oJT,\oFedevT,\otT)$.
Of course, the boundary conditions \eq{enhanced-BC} are to be
inherited appropriately, i.e.
\begin{subequations}\label{enhanced-BC-disc}\begin{align}
&\nn\Cdot\ovT=0\,,\ \ \ \ \ 
\big[\boldT(\uJT,\uFedevT,\utT){+}\oDTone{+}\oDTzero)\nn
+\divS(\bar{\mathfrak{H}}_\TAU\nn)\big]_\text{\sc t}^{}=\bm 0\,,\ \ \
\\&\bar{\mathfrak{H}}_\TAU\Colon(\nn{\otimes}\nn)=\bm0\,,\ \ \ \ 
(\nn\Cdot\nabla)\oLdT=\bm0\,,\ \ 
\ \text{ and }\ \ \ \nn\Cdot\nabla\otT=0\,.
\end{align}\end{subequations}

Besides the ansatz \eq{special-thermo-case+}, we accept the 
assumptions, for some $p,q>3$,
\begin{subequations}\label{EUL-L-Jeff-PT-ass}
\begin{align}\nonumber
&
\psi(J,F^*,\theta)=\varphi(F^*)+\wt\COUPLING(J,\theta)
\ \ \text{ with }\ 
\varphi\in C^1(
{\rm GL}_3^+),
\ \ \ \ \ \inf_{F^*\in{\rm SL}_3}\varphi(F^*)>-\infty\,,\ \text{ and} 
\\[-.3em]&\label{Euler-ass-therm-psi}
\hspace{6em}\text{with }\ \wt\COUPLING\in C^1(\R^+{\times}\R^+)
\,,\ \ \wt\COUPLING_\theta'(\cdot,0)=0\,,
\ \ \inf_{J\ge0,\ \theta\ge0}\wt\COUPLING(J,\theta)>-\infty\,,
\\[-.5em]&\nonumber
\forall\delta>0\ \
\exists\, 0<c_\delta^{}\le C_\delta^{}<+\infty
\ \ \forall(J,\theta)\in[\delta,1/\delta]{\times}\R^+:
\\[-.1em]&
\hspace{2em}\wt\ENT(J,\theta)\ge c_\delta^{}\theta\ \ \text{ and }\ \
\wt\ENT_\theta'(J,\cdot)>0\ \text{ on $(0,+\infty)$\,,}
\label{Euler-ass-therm-psi-1}
\\[-.1em]&
\hspace{2em}|\wt\ENT_J'(J,\theta)|+\theta\wt\ENT_\theta'(J,\theta)
\le C_\delta^{}\big(1\,{+}\,\theta\big)
\label{Euler-ass-therm-psi-1+++}
\ \text{ and }\ |\wt\ENT_{J\theta}''(J,\theta)|\le C_\delta^{}\big(1\,{+}\,\theta\big)
,
\\&\hspace{0em}\exists C<\infty\  \forall(J,F^*,\theta)\in\R^+{\times}
       {\rm SL}_3{\times}\R^+{:}
    \ |\boldT(J,F^*\!,0)|+|\ADI(J,\theta)|\le C\big(1{+}\wt\ENG(J,F^*,\theta)\big),
\label{ass-stress-control+}
\\
&\bbD_0\in{\rm Lin}(\Rsym,\Rsym)\,,
\ \
\bbD_1\in{\rm Lin}(\Rtr,\Rtr)\
\
\text{ positive definite}\,,
\ \
\HYPER_0,\HYPER_1>0
\,,\label{EUL-L-Jeff-St-ass-D-D}
\\&\nonumber\zetap\in C([0,1]{\times}\Rtr),\ \ \ \ 
\forall\chi\in[0,1]:\ \ \ \zetap(\chi,\cdot)\
\text{is convex smooth, and}\
\\[-.1em]&\hspace{5em}
(\chi,L)\mapsto[\zetap(\chi,\cdot)]'(L)
\ \text{ is continuous}\,,\ \ \ \sup\frac{|[\zetap(\chi,\cdot)]'(L)|}{1+|L|}<\infty\,,
\label{EUL-L-Jeff-St-thermo-ass-zetap}
\\[-.3em]&
\kappa\in C([0,1])\ \ \text{ positive}\,,\label{Euler-ass-therm-kappa+}
\\&\Upsilon\in C(\R)\ \text{ increasing},\ \Upsilon(0)=0,\ \
\sup|\Upsilon(\cdot)|/(1{+}|\cdot|^{1/2})<+\infty\,,\ \
\sup\Upsilon'(\cdot)<+\infty\,,
\label{ass-Upsilon}\\&\nonumber
\rhoR,\varrho_0\in W^{1,\infty}(\varOmega)\ \text{ positive}\,,\ \ \
\vv_0\in L^2(\varOmega;\R^3)\,,\ \ \
\Fezero\in W^{1,\infty}(\varOmega;\R^{3\times3}),\ \ \det\Fezero>0,
\ \
\\&
\qquad\theta_0\in L^1(\varOmega)\,,\ \theta_0\ge0\,,\ \ \ 
\chi_0\in W^{1,s}(\varOmega),\ 1\le s<5/4,\ \ 0\le\chi_0\le1\,,
\label{EUL-L-Jeff-St-thermo-ass-IC}\\&
\wt\ENG(\det\Fezero,\Fedevzero,\theta_0^{})\in L^1(\varOmega)\,,\ \ \ 
\GRAVITY\in L^1(I;L^\infty(\varOmega;\R^3))\,.
\label{EUL-L-Jeff-St-thermo-ass-load}\end{align}\end{subequations}
In (\ref{EUL-L-Jeff-PT-ass}a,b,f), ${\rm SL}_3=\{F\in\R^{3\times3};\ \det F=1\}$
denotes the special linear group while in \eq{Euler-ass-therm-psi}
${\rm GL}_3^+=\{F\in\R^{3\times3};\ \det F>0\}$
so that the derivative $\varphi'$ can be well defined even if restricted
on the nonlinear manifold ${\rm SL}_3$.
The condition $\COUPLING_\theta'(\cdot,0)=0$ in \eq{Euler-ass-therm-psi}
reflects the 3rd law of thermodynamics: the entropy at temperature $\theta=0$ 
is finite and independent of the mechanical state, standardly calibrated as 0.
Note that \eq{EUL-L-Jeff-St-ass-D-D}
ensures that $\DD_1$ from \eq{EUL-L-Jeff-St6-thermo-hyper-D1} is
$\Rdev$-valued and the contribution $\Feedev^\top\DD_1\Feedev^{-\top}$ to the
Mandel stress in \eq{EUL-L-Jeff-St2-thermo-grad} is traceless, albeit not
symmetric in general. The condition \eq{ass-stress-control+} is the
energy-controlled stress, here with the energy
$\wt\ENG(J,F^*,\theta)=\varphi(F^*)+\wt\COUPLING(J,\theta)
-\theta\wt\COUPLING_\theta'(J,\theta)$, cf.\
\eq{EUL-L-energy-Jeff-St-thermo}. This condition is well
complying with physically relevant situations, e.g.\ $\varphi$
with a polynomial growth or $\COUPLING(\cdot,\theta)$ with
a polynomial or logarithmic singularity at $J=0$, as 
it was articulated by J.M.\,Ball \cite{Ball84MELE,Ball02SOPE}
for the Kirchhoff rather than the Cauchy stress or the thermal pressure
as used here, cf.\ \cite[Remark~4.2]{Roub25TEFS}.
From \eq{ass-stress-control+}, we also obtain the control on
$\DEV\boldT(J,F^*,0)=\DEV(\varphi'(F^*)F^{*\top})$, which will be
used in \eq{enrg-est-of-stress3} below.

\begin{proposition}[Existence and regularity of weak solutions]\label{prop-Euler-PT}
Let $p>3$ and $q>3$ and the assumptions 
\eq{EUL-L-Jeff-PT-ass} hold. Then:\\
\Item{(i)}{For any $\THRESHOLD$ and for all sufficiently small time steps
$\tau>0$, the staggered scheme \eq{EUL-L-visco-thermo-PT+disc} has a solution
$(\varrho_\TAU^k,\vv_\TAU^k,\LpTAU^k,\FedevTAU^{\!\!k},\theta_\TAU^k,\chi_\TAU^k)\in
W^{1,r}(\varOmega)\times W^{2,p}(\varOmega;\R^3)\times W^{1,r}(\varOmega;\Rtr)\times
W^{1,q}(\varOmega;\R^{3\times3})\times W^{1,1}(\varOmega)\times W^{1,1}(\varOmega)$ and
$\varrho_\TAU^k>0$, $\det\FedevTAU^{\!\!k}=1$, and $\theta_\TAU^k>0$ a.e.\ on $\varOmega$ with
$1/\varrho_\TAU^k\in W^{1,r}(\varOmega)$.}
\ITEM{(ii)}
{For any $\THRESHOLD$ fixed and for $\tau\to0$, there is a
subsequence of 
$(\orT,\ovT,\oFedevT,\oLpT,\otT,\oxT)$ converging in the topologies specified in
\eq{EUL-L-converge} below to some limit $(\varrho,\vv,\Fedev,\Lp,\theta,\chi)$.
Moreover, there is a sufficiently large $\THRESHOLD$ for which any such five-tuple
is a weak solution in the sense of Definition~\ref{def-thermo-Ch5-PT}.}
\Item{(iii)}
  {the initial-boundary-value problem for the system
    \eq{EUL-L-Jeff-St6-thermo-hyper}--\eq{EUL-L-Jeff-St-thermo-grad} has
at least one weak solution $(\varrho,\vv,\Fedev,\Lp,\theta,\chi)$ in the sense
of Definition~\ref{def-thermo-Ch5-PT} such that also
$\varrho\in L^\infty(I;W^{1,r}(\varOmega))\,\cap\,C(I{\times}\barOmega)$ with
$\min_{I{\times}\barOmega}\varrho>0$, $\vv\in L^\infty(I;L^2(\varOmega;\R^3))$,
$\Fedev\in L^\infty(I;W^{1,r}(\varOmega;\R^{3\times3}))
\,\cap\,C(I{\times}\barOmega;\R^{3\times3})$ with
$\det\Fedev=1$, and further $\theta\in C_{\rm w}(I;L^{1+\ALPHA}(\varOmega))
\,\cap\,L^\EXP(I;W^{1,\EXP}(\varOmega))$ with 
$1\le\EXP<5/4$, and $\chi\in L^\infty(I;W^{1,1}(\varOmega)$.}
\end{proposition}

\noindent{\it Proof.}
The main conceptual aspect is that, from the energetics, we can read the a-priori
estimates for $\vv\in L^p(I;W^{2,p}(\varOmega;\R^{3\times3}))$ and
$\Ld\in L^2(I;W^{1,q}(\varOmega;\Rtr))$. Then, from 
$\pdt{}\Fedev=\Ld\Fedev-(\vv\Cdot\nabla)\Fedev$ with the initial
condition $\Fedev|_{t=0}^{}\in W^{1,q}(\varOmega;\R^{3\times3})$
with $\det\Fedev|_{t=0}^{}=1$, we can read also
$\Fedev\in L^\infty(I;W^{1,q}(\varOmega;\R^{3\times3}))$. We will split the proof
to four steps. Some points are rather sketched while referring to
the  analysis of the time discretization of a non-degenerate
variant of the the Jeffreys rheology in \cite{Roub26TDCL,Roub25STDF}.

\medskip{\it Step 1: The choice of  $\THRESHOLD$}. 
We analyze
the (now formal) energetics of the model, departing from the total
energy balance \eq{EUL-L-energy-Jeff-St-thermo}. From this, also by
the Gronwall-inequality strategy 
exploiting also the total-mass conservation due to the equation
\eq{EUL-L-Jeff-St0-thermo}, we can read the bound for the energy
$\int_\varOmega\wt\ENG(J,\Fedev,\theta)\,\d\xx\le\THRESHOLD-1$ uniformly in time $t\in I$ 
with $\wt\ENG(J,\Fedev,\theta)=\varphi(\Fedev)+\wt\ENT(J,\theta)
$. This gives the suggestion for the choice of $\THRESHOLD$.

\medskip{\it Step 2: Basic stability of the scheme
\eq{EUL-L-visco-thermo-PT+disc} and first a-priori estimates.}
Exploiting the convexity of $(\varrho,\pp)\mapsto|\pp|^2/\varrho$ 
and using the discrete variant of the calculus \eq{EUL-L-Jeff-St-calcul+}, 
we obtain
\begin{align}\nonumber
  \!\!\!\!\!&\!\!\int_\varOmega\!\frac{|\pp_\TAU^k|^2\!\!\!}{2\varrho_\TAU^k\!\!\!}\,\d\xx
  +\tau\!\sum_{m=1}^k\int_\varOmega\!\bigg(
\bbD_0\strain(\vv_\TAU^m)\Colon\strain(\vv_\TAU^m)+\bbD_1\LdTAUm\Colon\LdTAUm
+\HYPER_0|\nabla\strain(\vv_\TAU^m)|^p
\\[-1.2em]&\nonumber\hspace{12em}
+\HYPER_1
|\nabla\LdTAUm|^q
+[\zetap(
\chi_\TAU^{m-1};\cdot)]'(\LpTAU^m)\Colon\LpTAU^m\bigg)\,\d\xx
\\[-.8em]&\nonumber\hspace{.2em}
\ \le\!\int_\varOmega\!\frac{|\pp_0|^2\!\!}{2\varrho_0\!\!}
\,\d\xx
+\tau\!\sum_{m=1}^k\int_\varOmega\!\bigg(\varrho_\TAU^m\GRAVITY_{\tau}^m\Cdot\vv_\TAU^m
\\[-.3em]&\hspace{.8em}
-\frac{\boldT(J_\TAU^{m-1}\!,\FedevTAU^{\!\!m-1})\Colon\nabla\vv_\TAU^m
+\ADI(J_\TAU^{m-1}\!,\theta_\TAU^{m-1})\,{\rm div}\,\vv_\TAU^m
+\boldM(\FedevTAU^{\!\!m-1})\Colon\LpTAU^m}{\!\!
1+\big(\|\wt\ENG(J_\TAU^{m-1}\!,\FedevTAU^{\!\!m-1}\!,\theta_\TAU^{m-1})\|_{L^1(\varOmega)}{-}
\THRESHOLD\big)^+}
\bigg)\,\d\xx\,.\!\!\!
\label{EUL-L-basic-engr-disc+}\end{align}

The last term in \eq{EUL-L-basic-engr-disc+}
can be estimated by the H\"older inequality,
we have
\begin{subequations}\begin{align}\nonumber
\!\!\int_\varOmega&
\frac{-\,
\boldT(J_\TAU^{m-1},\FedevTAU^{\!\!m-1})\Colon\nabla\vv_\TAU^m
}{
1+\big(\|\wt\ENG(J_\TAU^{m-1},\FedevTAU^{\!\!m-1},\theta_\TAU^{m-1})\|_{L^1(\varOmega)}{-}
\THRESHOLD\big)^+}
\,\d\xx
\\[-.3em]&\ \nonumber
=-\,\frac1{\!\!1+\big(\|\wt\ENG(J_\TAU^{m-1},\FedevTAU^{\!\!m-1}\!,\theta_\TAU^{m-1})\|_{L^1(\varOmega)}{-}
\THRESHOLD\big)^+}\int_\varOmega
\boldT(J_\TAU^{m-1},\FedevTAU^{\!\!m-1})\Colon\strain(\vv_\TAU^m)\,\d\xx
\\&\ 
\le\frac{\|\boldT(J_\TAU^{m-1},\FedevTAU^{\!\!m-1})\|_{L^1(\varOmega;\R^{3\times3})}
\|\strain(\vv_\TAU^m)\|_{L^\infty(\varOmega;\R^{3\times3})}^{}
}
{\!\!1+\big(\|\wt\ENG(J_\TAU^{m-1},\FedevTAU^{\!\!m-1}\!,\theta_\TAU^{m-1})\|_{L^1(\varOmega)}{-}
\THRESHOLD\big)^+}
\le\Frac23C(1{+}\THRESHOLD)\|\strain(\vv_\TAU^m)\|_{L^\infty(\varOmega;\R^{3\times3})}^{}
\nonumber\\[-1.3em]\end{align}
with $C$ from \eq{ass-stress-control+}. Similarly,
\begin{align}\nonumber
\!\!\int_\varOmega&
\frac{-\,\ADI(J_\TAU^{m-1},\theta_\TAU^{m-1})\,{\rm div}\,\vv_\TAU^m}{
1+\big(\|\wt\ENG(J_\TAU^{m-1},\FedevTAU^{\!\!m-1},\theta_\TAU^{m-1})\|_{L^1(\varOmega)}{-}
\THRESHOLD\big)^+}
\,\d\xx
\\&\ 
\le\frac{\|\ADI(J_\TAU^{m-1}\!,\theta_\TAU^{m-1})\|_{L^1(\varOmega)}
\|{\rm div}\,\vv_\TAU^m\|_{L^\infty(\varOmega)}^{}
}
{\!\!1+\big(\|\wt\ENG(J_\TAU^{m-1},\FedevTAU^{\!\!m-1}\!,\theta_\TAU^{m-1})\|_{L^1(\varOmega)}{-}\THRESHOLD\big)^+}
\le\Frac23C(1{+}\THRESHOLD)\|{\rm div}\,\vv_\TAU^m\|_{L^\infty(\varOmega)}^{}
\end{align}
with $C$ again from \eq{ass-stress-control+}. Eventually, using
the matrix calculus $\DEV(F^\top S)\Colon(F^{-1}LF)=
(F^\top S-\frac13{\rm tr}(F^\top S))\Colon(F^{-1}LF)=
(SF^\top-\frac13{\rm tr}(F^\top S))\Colon L=
(SF^\top-\frac13{\rm tr}(S F^\top))\Colon L=\DEV(SF^\top)\Colon L$ for
$S=\varphi'(\FedevTAU^{\!\!m-1})$, $F=\FedevTAU^{\!\!m-1}$, and
$L=\nabla\vv_\TAU^m{-}\LdTAUm$, we obtain the estimate
\begin{align}\nonumber
&\!\!\int_\varOmega
\frac{-\,\boldM(\FedevTAU^{\!\!m-1})\Colon\LpTAU^m
}{
1+\big(\|\wt\ENG(J_\TAU^{m-1},\FedevTAU^{\!\!m-1},\theta_\TAU^{m-1})\|_{L^1(\varOmega)}{-}
\THRESHOLD\big)^+}
\,\d\xx
\\[-.3em]&\ \nonumber
=\!\frac{\int_\varOmega
\DEV(\varphi'(\FedevTAU^{\!\!m-1})(\FedevTAU^{\!\!m-1})^\top)
\Colon(\LdTAUm{-}\nabla\vv_\TAU^m)\,\d\xx}
  {\!\!1+\big(\|\wt\ENG(J_\TAU^{m-1}\!,\FedevTAU^{\!\!m-1}\!,\theta_\TAU^{m-1})\|_{L^1(\varOmega)}{-}\THRESHOLD\big)^+}
\\&\ \nonumber
\le\frac{\|\DEV\boldT(\FedevTAU^{\!\!m-1})\|_{L^1(\varOmega;\R^{3\times3})}
\|\nabla\vv_\TAU^m{-}\LdTAUm\|_{L^\infty(\varOmega;\R^{3\times3})}^{}
}
{\!\!1+\big(\|\wt\ENG(J_\TAU^{m-1}\!,\FedevTAU^{\!\!m-1}\!,\theta_\TAU^{m-1})\|_{L^1(\varOmega)}{-}\THRESHOLD\big)^+}
\le\Frac23C(1{+}\THRESHOLD)
\|\nabla\vv_\TAU^m{-}\LdTAUm\|_{L^\infty(\varOmega;\R^{3\times3})}^{}\,.
\nonumber\\[-1.3em]\label{enrg-est-of-stress3}
\end{align}\end{subequations}
The gravity term $\varrho_\TAU^m\GRAVITY_{\tau}^m\Cdot\vv_\TAU^m$ in
\eq{EUL-L-basic-engr-disc+} can be estimated by exploiting the
bound $\int_\varOmega\varrho_\TAU^m\,\d\xx=\int_\varOmega\varrho_0\,\d\xx<\infty$.
Therefore, from \eq{EUL-L-basic-engr-disc+} we obtain the a-priori estimates
\begin{subequations}\label{est-of-basic}\begin{align}\label{est-of-p/sqrt-disc}
&\|\opT/\sqrt{\orT}\|_{L^\infty(I;L^2(\varOmega;\R^3))}^{}\le C\,,
    \\&\label{est-of-v-disc}
    \|\strain(\ovT)\|_{L^2(I;W^{1,p}(\varOmega;\R^{3\times3}))}^{}\le C\,,
    \ \text{ and}
    \\&\label{est-of-Le-disc}
 \|\oLdT\|_{L^2(I\times\varOmega;\R^{3\times3})\,\cap\,L^q(I;W^{1,q}(\varOmega;\R^{3\times3}))}^{}\le C\,. 
\end{align}\end{subequations}

Having $\nabla\ovT$ estimated so that, in particular, due to $p>3$, $\ovT(t)$
is Lipschitz continuous on $\varOmega$ integrably in time,
by the discrete Gronwall inequality for sufficiently small time
steps $\tau>0$, we also obtain the estimate for the mass density $\orT$ and
for the sparsity $\osT=1/\orT$:
\begin{subequations}\label{est-of-basic+}\begin{align}\label{est-of-rho-disc}
  &\|\orT\|_{L^\infty(I;W^{1,r}(\varOmega))}^{}\le C_r\ \ \text{ and }\ \  
    \|\osT\|_{L^\infty(I;L^r(\varOmega))}^{}\le C_r
\,;
\intertext{here we used also that $\orT>0$ a.e., cf.\ \cite{Roub26TDCL,Roub25STDF}.
Moreover, from $\nabla\osT=\nabla(1/\orT)=-(\nabla\orT)/\orT^2
=-\osT^2\nabla\orT$, we also obtain}
&\|\osT\|_{L^\infty(I;W^{1,r}(\varOmega))}^{}\le C_r\ \ \text{ for any }
1\le r<\infty
\,,
\label{est-of-sigma-disc}
\intertext{from which we can also see that $\min\orT=1/\max\osT>0$.
From this, in fact we can improve \eq{est-of-sigma-disc} for $a=r$.
As a result, using also \eq{est-of-p/sqrt-disc}, we have also the
estimates for $\opT=\sqrt{\osT}(\opT/\sqrt{\orT})$ and for $\ovT=\osT\opT$
 as}
&\|\opT\|_{L^\infty(I;L^2(\varOmega;\R^3))}^{}\le C
\ \ \text{ and }\ \  
\|\ovT\|_{L^\infty(I;L^2(\varOmega;\R^3))}^{}\le C\,.
\label{est-of-p-v-disc}
\intertext{From this together with the estimate \eq{est-of-v-disc}, by the
Korn inequality
we can also read the bound of $\ovT$ in $L^2(I;W^{2,p}(\varOmega;\R^3))$.
This also allows us to estimate  $\nabla\opT=
\nabla(\orT\ovT)=\orT\nabla\ovT+\nabla\orT{\otimes}\ovT$, namely}
\label{est-of-p-disc}
&\|\opT\|_{L^{p}(I;W^{1,r}(\varOmega;\R^3))}^{}\le C\ \
\text{ and, by comparison, }\
\Big\|\pdt{\pp_\TAU}\Big\|_{L^{\min(p,q)}(I;W^{2,p}(\varOmega;\R^3))}\!\le C\,.
\intertext{By comparison, from \eq{EUL-L-visco-thermo-PT+0disc+}, we thus have}
\label{est-of-dt-rho-disc}
&\Big\|\pdt{\varrho_\TAU}\Big\|_{L^p(I;L^r(\varOmega))}^{}\le C_{p,r}\,.
\intertext{From (\ref{est-of-basic}b,c), using the kinematic flow-rule
\eq{EUL-L-visco-thermo-PT+2disc+} together from the regularity of the initial
condition
  $\Fedevzero\in W^{1,q}(\varOmega;\R^{3\times3})$,
  we also obtain}
\label{EUL-L-basic-engr-disc-dev-Ge-EPS}
&\|\oFedevT\|_{L^\infty(I;W^{1,\min(p,q)}(\varOmega;\R^{3\times3}))}^{}\le C\ \
\text{ and, by comparison, }\
\Big\|\pdt{\FedevTAU}\Big\|_{L^2(I{\times}\varOmega;\R^{3\times3}))}\!\le C\,.
\intertext{Since $\zetap(\chi,\cdot)$ can degenerate
for temperatures above the melting point
in the liquid phase (i.e.\ at $\chi=0$), we cannot read any estimate
for $\oLpT$ directly from the dissipative $\zetap$-term
in \eq{EUL-L-basic-engr-disc+}. Yet, from
$\oLpT=\uFedevT^{\!\!\!-1}\DEV(\nabla\ovT-\oLdT)\uFedevT$, we can also read}
&\|\oLpT\|_{L^2(I;W^{1,\min(p,q)}(\varOmega;\R^{3\times3}))}^{}\le C\,.
\label{EUL-L-basic-engr-disc-dev-Lp-EPS}
\end{align}\end{subequations}

Testing  \eq{EUL-L-visco-thermo-PT+3disc} by 1 and summing it over time
steps, we obtain
\begin{align}&\nonumber
  \int_\varOmega\wt\ENT(J_\TAU^k,\theta_\TAU^k)+\varrho_\TAU^k\ell\chi_\TAU^k\,\d\xx
  =\int_\varOmega\wt\ENT(J_0,\theta_0)+\varrho_0\ell\chi_0\,\d\xx
  \\[-.5em]&\quad
  +\,\tau\!\sum_{m=1}^k\int_\varOmega\xi(
  \chi_\TAU^m;\vv_\TAU^m,\LdTAUm,\LpTAU^m)
+\frac{\ADI(J_\TAU^m,\theta_\TAU^m)\,{\rm div}\,\vv_\TAU^m}{\!
1+\Big(\|\wt\ENG(J_\TAU^m,\FedevTAU^{\!\!m},\theta_\TAU^m)\|_{L^1(\varOmega)}{-}\THRESHOLD\Big)^+\!}\,\d\xx\,.
\end{align}
As the right-hand side is already estimated due to
(\ref{est-of-basic}b,c) and \eq{EUL-L-basic-engr-disc-dev-Lp-EPS},
\begin{align}
\|\otT\|_{L^\infty(I;L^1(\varOmega))}^{}\le C\,.
\label{est-of-basic-theta}\end{align}

Furthermore, analyzing \eq{EUL-L-visco-thermo-PT+4disc}, we obtain immediately
$\oxT\in L^\infty(I{\times}\varOmega)$. Testing \eq{EUL-L-visco-thermo-PT+4disc}
by the (discrete) convective derivative
$(\chi_\TAU^k{-}\chi_\TAU^{k-1})/\tau+\vv_\TAU^k\Cdot\nabla\chi_\TAU^k$, 
we obtain its estimate in $L^2(I{\times}\varOmega)$. More in detail,
\begin{align}\nonumber
\!\!\!\!\!\!\int_\varOmega\!\NU&\Big|\pdt{\chi_\TAU}+\ovT\Cdot\nabla\oxT\Big|^2\!
+H^{-1}(\oxT)\pdt{\chi_\TAU}+H^{-1}(\oxT)\ovT\Cdot\nabla\oxT\,\d\xx
\\[-.3em]&\nonumber
=\!\!\int_\varOmega\!\!\Upsilon(\otT{-}\theta_\text{\sc pt})\Big(\pdt{\chi_\TAU\!}
+\ovT\Cdot\nabla\oxT\Big)\,\d\xx
\le\big\|\Upsilon(\otT{-}\theta_\text{\sc pt})\big\|_{L^2(\varOmega)}^{}
\bigg\|\pdt{\chi_\TAU\!}+\ovT\Cdot\nabla\oxT\bigg\|_{L^2(\varOmega)}^{}
\\[-.3em]&
\le\Frac1{2\NU}
\big\|\Upsilon(\otT{-}\theta_\text{\sc pt})\big\|_{L^2(\varOmega)}^2
+
\Frac{\NU}2\Big\|\pdt{\chi_\TAU\!}+\ovT\Cdot\nabla\oxT\Big\|_{L^2(\varOmega)}^2.\!\!
\label{EUL-L-est-DT-chi-disc}\end{align}
Here we realize that $H^{-1}(\oxT)\pdt{}\chi_\TAU=\pdt{}\delta_{[0,1]}^{}(\oxT)=0$
and $H^{-1}(\oxT)\ovT\Cdot\nabla\oxT=\ovT\Cdot\nabla\delta_{[0,1]}^{}(\oxT)=0$
because $\oxT$ is valued in $[0,1]$. Here we used also that
$\Upsilon(\otT{-}\theta_\text{\sc pt})$ is bounded in $L^2(I{\times}\varOmega)$
due to the already proved bound 
\eq{est-of-basic-theta} and the growth assumption \eq{ass-Upsilon} on $\Upsilon$.
Therefore, integrating \eq{EUL-L-est-DT-chi-disc} over $I=[0,T]$, we have
\begin{align}\nonumber\\[-2.7em]
\Big\|\pdt{\chi_\TAU}+\ovT\Cdot\nabla\oxT\Big\|_{L^2(I\times\varOmega)}^{}\le C\,.
\label{est-DT-chi}\end{align}

This information can now be used for the estimation of $\nabla\otT$ by testing 
\eq{EUL-L-visco-thermo-PT+3disc} by $\omega(\otT):=1-1/(1{+}\otT)^a$ with
$a>0$. To this aim, we rewrite \eq{EUL-L-visco-thermo-PT+3disc} as
\begin{align}
&\nonumber
\frac{\wt\Ent_\TAU^k{-}\wt\Ent_\TAU^{k-1}\!\!}\tau\,
=\frac{\ADI(J_\TAU^k,\theta_\TAU^k)\,{\rm div}\,\vv_\TAU^k}{\!
  1{+}\big(\|\wt\Eng_\TAU^k
  \|_{L^1(\varOmega)}{-}\THRESHOLD\big)^+\!}
+{\rm div}\big(\kappa(
\chi_\TAU^k)\nabla\theta_\TAU^k{-}\wt\Ent_\TAU^k\vv_\TAU^k\big)
+\xi(\chi_\TAU^k;\vv_\TAU^k,\LdTAUk,\LpTAU^k)
-\ell \chi_\TAU^k\varrho_\TAU^k{\rm div}\,\vv_\TAU^k
\\[-.2em]&\hspace{.5em}
-\ell
\varrho_\TAU^k\Big(\frac{\chi_\TAU^k{-}\chi_\TAU^{k-1}\!\!}\tau
+\vv_\TAU^k\Cdot\nabla\chi_\TAU^k\Big)
-\ell
\Big(\chi_\TAU^{k-1}\frac{\varrho_\TAU^k{-}\varrho_\TAU^{k-1}\!\!}\tau
+\chi_\TAU^k\vv_\TAU^k\Cdot\nabla\varrho_\TAU^k\Big)
\ \ \ \text{with }\ \ \wt\Ent_\TAU^k=
\wt\ENT(J_\TAU^k,\theta_\TAU^k)\,.
\label{EUL-L-visco-thermo-PT+3modif}
\end{align}
The $\ell$-terms in \eq{EUL-L-visco-thermo-PT+3modif} are already estimated
due to the already obtained estimates
\eq{est-of-rho-disc}, \eq{est-of-dt-rho-disc}, and \eq{est-DT-chi}, specifically
\begin{align}
  &\nonumber
  \Big\|\oxT\orT{\rm div}\,\ovT+\orT\Big(\pdt{\chi_\TAU\!}\,{+}\ovT\Cdot\nabla\oxT\Big)
+\uxT\!\pdt{\varrho_\TAU}+\oxT\ovT\Cdot\nabla\orT\Big\|_{L^1(I\times\varOmega)}^{}
\\&
\le\|\orT\|_{L^2(I\times\varOmega)}^{}\Big(\|{\rm div}\,\ovT\|_{L^2(I\times\varOmega)}^{}\!
+\Big\|\pdt{\chi_\TAU\!}{+}\ovT\Cdot\nabla\oxT\Big\|_{L^2(I\times\varOmega)}^{}\Big)\!
+\Big\|\pdt{\varrho_\TAU}\Big\|_{L^1(I\times\varOmega)}^{}\!\!+
\|\ovT\Cdot\nabla\orT\|_{L^1(I\times\varOmega)}^{}.
\nonumber\end{align}
Also the $\ADI$- and the $\xi$-terms \eq{EUL-L-visco-thermo-PT+3modif} are estimated
in $L^1(I{\times}\varOmega)$. If $\wt\ENT=\wt\ENT(\theta)$ would be $J$-independent,
then, as in \cite{Roub23SPTC}, we obtain the estimate 
\begin{align}\label{est-W-eps}
\|\otT\|_{L^{5/4-\epsilon}(I;W^{1,5/4-\epsilon}(\varOmega))\,\cap\,L^{5/3-\epsilon}(I{\times}\varOmega)}^{}\le C_\epsilon\ \ \text{ for any }\ 0<\epsilon\le1/4\,
\end{align}
by a sophisticate use of the Gagliardo–Nirenberg, cf.\ also
\cite[Proposition 8.2.1]{KruRou19MMCM}). The general case when $\wt\ENT$ is
$J$-dependent, or in other words the $\varrho$-dependence of
$\wt\ENT=\wt\ENT(\rhoR/\varrho,\theta)$, this estimation is to enhanced as in
\cite{Roub25TEFS} for the regularized time-continuous problem, relying in particular
on the estimate of $\pdt{}\varrho_\tau$ in \eq{est-of-dt-rho-disc}.
Here we also use that, due to (\ref{est-of-basic+}a,b), $\oJT=\rhoR/\orT$ is
valued in $[\delta,1/\delta]$ to be used in \eq{Euler-ass-therm-psi-1}.

\def\etau{\TAU}

Now, testing \eq{EUL-L-visco-thermo-PT+4disc} by
${\rm div}(|\nabla\chi_\TAU^k|^{r-2}\nabla\chi_\TAU^k)$ with $1\le r<5/4$, we can estimate
\begin{align}\nonumber
&\Frac1r\int_\varOmega\frac{|\nabla\chi_\TAU^k|^r-|\nabla\chi_\TAU^{k-1}|^r}{\tau}
  \le\int_\varOmega\Big(\frac{\nabla\Upsilon(\theta_\TAU^k{-}\theta_{\rm pt})}\NU
  -\vvk{\cdot}\nabla\chi_\TAU^k\Big)\cdot|\nabla\chi_\TAU^k|^{r-2}\nabla\chi_\TAU^k\,\d\xx
\\&\nonumber\quad=\int_\varOmega\!\frac{\Upsilon'(\theta_\TAU^k{-}\theta_{\rm pt})}{\NU}\nabla\theta_\TAU^k\Cdot|\nabla\chi_\TAU^k|^{r-2}
\nabla\chi_\TAU^k+|\nabla\chi_\TAU^k|^{r-2}\big(\nabla\chi_\TAU^k{\otimes}\nabla\chi_\TAU^k\big)
\Colon\strain(\vvk)-\Frac1r|\nabla\chi_\TAU^k|^q{\rm div}\,\vvk\,\d\xx
\\&\hspace{9em}\le\frac{\sup\Upsilon'}{\NU}\|\nabla\theta_\TAU^k\|_{L^r(\varOmega;\R^3)}^r\!+
\big(1{+}\|\strain(\vvk)\|_{L^\infty(\varOmega;\R^{3\times3})}^{}\big)
\|\nabla\chi_\TAU^k\|_{L^r(\varOmega;\R^3)}^r\,,
\label{grad-chi}
\end{align}
where we used $\nn\Cdot\vv_\TAU^k=0$. The first inequality in \eq{grad-chi}
follows from the convexity of $|\cdot|^r$ and from that, written formally,
$\nabla H^{-1}(\chi_\etau^k)\cdot|\nabla\chi_\etau^k|^{r-2}\nabla\chi_\etau^k
=\partial^2\delta_{[0,1]}^{}(\chi_\etau^k)|\nabla\chi_\etau^k|^r\ge0$, where
$\partial^2\delta_{[0,1]}^{}$ denotes the (generalized) Hessian of the convex
indicator function of the interval $[0,1]$. Therefore, we obtain the estimate
\begin{subequations}\label{est-of-derivatives-chi}
  \begin{align}\label{est-of-derivatives-chi-1}
    &\|\nabla\oxT\|_{L^\infty(I;L^{5/4-\epsilon}(\varOmega;\R^3))}^{}\le C\ \ \
    \text{ with }\ 0<\epsilon\le1/4\,.
\intertext{Exploiting the calculus
$\nabla\ouT=\nabla(\wt\ENT(\oJT,\otT)+\orT\ell\oxT)
=\wt\ENT_J'(\oJT,\otT)\nabla\oJT+\wt\ENT_\theta'(\oJT,\otT)\nabla\otT
+\orT\ell\nabla\oxT+\ell\oxT\nabla\orT$
with $\nabla\oJT$ bounded in $L^\infty(I;L^r(\varOmega;\R^3))$ for any $1\le r<+\infty$
and relying on \eq{est-of-derivatives-chi-1} and \eq{est-W-eps}, we have also the
bound 
}
  &\|\ouT\|_{L^\infty(I;L^1(\varOmega))\,\cap\,
      L^{5/4-\epsilon}(I;W^{1,5/4-\epsilon}(\varOmega))\,\cap\,L^{5/3-\epsilon}(I\times\varOmega)}^{}\le C_\epsilon\ \
    \text{ with $\ 0<\epsilon\le1/4$}\,.
\intertext{From \eq{est-DT-chi}, realizing the bound of
$\ovT\Cdot\nabla\oxT$ in $L^p(I;L^s(\varOmega))$, we can now read also the estimate}
&\Big\|\pdt{\chi_\TAU}\Big\|_{L^2(I;L^{5/4-\epsilon}(\varOmega))}^{}\le C\,.
\intertext{Eventually, by comparison, from \eq{EUL-L-visco-thermo-PT+3disc+}
  when exploiting the estimates \eq{est-of-v-disc}, \eq{est-of-p-v-disc}, and
  \eq{est-W-eps} leading in particular to the bound of $\ouT\ovT$ surely in
$L^{15/14}(I;L^{5/4}(\varOmega;\R^3))$, we obtain also}
&\Big\|\pdt{\Ent_\TAU}\Big\|_{L^1(I;W^{1,5}(\varOmega)^*)}^{}\le C\,.
\end{align}\end{subequations}

\medskip{\it Step 3: Convergence for $\tau\to0$ in the mechanical part}.
By the Banach selection theorem, we can choose a subsequence con\-verging to
some limit $(\varrho,\sigma,\pp,\vv,\Ld,\Fedev,J,\Ent,\chi,\theta)$
in the sense \begin{subequations}\label{EUL-L-converge}
\begin{align}
&\orT\to\varrho\ \text{ and }\ \urT\to\varrho
\!\!\!\hspace{-7.5em}&&\hspace{3em}\text{weakly* in $\ L^\infty(I;W^{1,r}(\varOmega))$\,,}
\\&\varrho_\tau\to\varrho&&\text{weakly* in
$\ L^\infty(I;W^{1,r}(\varOmega))\,\cap\,W^{1,p}(I;L^r(\varOmega))$\,,}
\\&\label{EUL-L-converge-bar-s}
\osT\to\sigma&&\text{weakly* in $\ L^\infty(I;W^{1,r}(\varOmega))$\,,}
\\&\label{EUL-L-converge-bar-p}
\opT\to\pp&&\text{weakly\ \;in $\ L^p(I;W^{1,r}(\varOmega;\R^3))$\,,}
\\&\label{EUL-L-converge-p}
\pp_\tau\to\pp&&\text{weakly\ \;in $\
L^p(I;W^{1,r}(\varOmega;\R^3))\,\cap\,W^{1,p'}(I;W^{2,p}(\varOmega;\R^3)^*)$\,,}
\\&\label{EUL-L-converge-bar-v}
\ovT\to\vv&&\text{weakly\ \;in $\ L^a(I;W^{2,p}(\varOmega;\R^3))$\ \ with any $\,a>2$\,,}
\\&\oLdT\to\Ld&&\text{weakly in $\ L^2(I;W^{1,r}(\varOmega;\Rtr))$\ with $1\le r<q$\,,}
\label{EUL-L-converge-Ld-weak}
\\[-.1em]
&\tildeFedevTAU\to\Fedev,\ \oFedevT\to\Fedev\!,
\text{ and }\ \uFedevT\to\Fedev\!\ \!\!\!\hspace{-11em}
&&\hspace{11em}\text{weakly* in
$\,L^\infty(I;W^{1,\min(p,q)}(\varOmega;\R^{3\times3}))$,}\!\!
\label{EUL-L-converge-G-weak}\\
&\oJT\to J\text{ and }\ \uJT\to J
\!\!\!\hspace{-7.5em}&&\hspace{3em}\text{weakly* \,in
$\,L^\infty(I;W^{1,r}(\varOmega))$;}\!\!
\\&\otT\to\theta\ \text{ and }\ \utT\to\theta\hspace{-7.5em}&&\hspace{3em}\text{weakly* \,in $L^{5/4-\epsilon}(I;W^{1,5/4-\epsilon}(\varOmega))\,\cap\,L^{5/3-\epsilon}(I{\times}\varOmega)$;}\!\!
\\
&\oxT\to\chi\text{ and }\ \uxT\to\chi\!\!\!\hspace{-7.5em}&&\hspace{3em}\text{weakly* in $\ L^\infty(I{\times}\varOmega)\,\cap\,
  L^p(I;W^{1,5/4-\epsilon}(\varOmega))$
  \,,}
\\&\ouT\to \Ent\ \text{ and }\ \uuT\to\Ent\hspace{-7.5em}
&&\hspace{3em}
\text{weakly* \,in $L^{5/4-\epsilon}(I;W^{1,5/4-\epsilon}(\varOmega))\,\cap\,L^{5/3-\epsilon}(I{\times}\varOmega)$;}\!\!
\end{align}\end{subequations}
here the exponent $r$ in (\ref{EUL-L-converge}a-e) can be
arbitrarily large. Notably, the limit of $\orT$, $\urT$, and $\varrho_\tau$ is indeed the
same due to the control of $\pdt{}\varrho_\tau$ in $L^p(I;L^r(\varOmega))$ by 
comparison of \eq{EUL-L-visco-thermo-PT+0disc+} using the former
estimate in \eq{est-of-p-disc}; cf.\
\cite[Sect.8.2]{Roub13NPDE}. Similarly, the same is true also for $\opT$ and
$\pp_\tau$, and for $\oFedevT$ and $\FF_{{\rm e},\tau}$, etc.

By the (generalized) Aubin-Lions theorem and the Arzel\`a-Ascoli-type theorem
\cite[Sect.7.3]{Roub13NPDE} and the information about time derivatives
from Step~2, we have
\begin{subequations}\label{EUL-L-converge-strong}
\begin{align}\label{EUL-L-converge-strong-rho}
&&&\!\!\orT\to\varrho\ \ \text{ and }\ \ \urT\to\varrho\!\!\!\hspace{-5em}&&\hspace{5em}\text{strongly in }\ L^\infty(I{\times}\varOmega)\,,&&&&
\\\label{EUL-L-converge-strong-p-}
&&&\!\!\opT\to\pp\!\!\!\!\!\!&&\text{strongly in }\ L_{\rm w}^p(I;L^\infty(\varOmega;\R^3))\,,
\\&&&\!\!\oFedevT\!\to\Fedev\ \text{ and }\,\ \uFedevT\!\to\Fedev\!\ \!\!\!\hspace{-5em}
&&\hspace{5em}\text{strongly in $\,L^\infty(I{\times}\varOmega;\R^{3\times3})$.}\!\!
\label{EUL-L-converge-G-strong}
\\&&&\!\!\oxT\to\chi\ \ \text{ and }\ \ \uxT\to\chi\!\!\!\hspace{-5em}&&\hspace{5em}\text{strongly in }\ L^\infty(I{\times}\varOmega)\,,\ \text{ and }&&
\label{EUL-L-converge-strong-chi}
\\&&&\!\!\ouT\to\Ent\ \ \;\text{ and }\ \ \uuT\to\Ent\!\!\!\hspace{-5em}&&\hspace{5em}\text{strongly in }\ L^{5/3-\epsilon}(I{\times}\varOmega)\,.&&
\label{EUL-L-converge-strong-u}
\end{align}\end{subequations}
Since $\inf_{\TAU>0}\inf_{I\times\varOmega}\orT>0$, we have also
\begin{subequations}\begin{align}
&\oJT\to J\ \text{ and }\ \uJT\to J=\rhoR/\varrho
  &&\hspace{-1em}\text{strongly in }\ L^\infty(I{\times}\varOmega)
  \ \text{ and}
\label{EUL-L-converge-J-strong}
\\&\ovT={\opT}/{\orT}\to{\pp}/{\varrho}=\vv\!\!\!
&&
\label{Euler-small-converge-strong-vv}
\hspace{-1em}\text{strongly in 
  $\ L_{\rm w}^p(I;L^\infty(\varOmega;\R^3))$\,.}
\intertext{By the continuity and the uniform convergence of $[\wt\ENT(\oJT,\cdot)]^{-1}$
  and $[\wt\ENT(\uJT,\cdot)]^{-1}$ and by 
(\ref{EUL-L-converge-strong}a,d,e) and \eq{EUL-L-converge-J-strong}, we still obtain}
&\!\!\otT\to\theta\ \ \;\text{ and }\ \ \utT\to\theta
=\big[\wt\ENT(J,\cdot)\big]^{-1}(\Ent{-}\varrho\ell\chi)
\!\!\!\hspace{-5em}&&\hspace{3em}\text{strongly in }\ L^{5/3-\epsilon}(I{\times}\varOmega)\,.
\end{align}\end{subequations}
The strong convergence like \eq{EUL-L-converge-strong-u} holds also for
$\oeT=\wt\ENG(\oJT,\oFedevT,\otT)
=\varphi(\oFedevT)+\COUPLING(\oJT,0)+\wt\ENT(\oJT,\otT)$ and
$\ueT=\wt\ENG(\uJT,\uFedevT,\utT)$.

Moreover, we need also the strong convergence of $\nabla\ovT$. This can be
obtained by the interpolation in between the strong convergence
\eq{Euler-small-converge-strong-vv} and the weak convergence
\eq{EUL-L-converge-bar-v}. By the Gagliardo-Nirenberg
inequality, $\|\nabla\vv\|_{L^\infty(\varOmega;\R^{3\times3})}^{}
\le N \|\vv\|_{L^\infty(\varOmega;\R^3)}^{1/2}\|\nabla\vv\|_{W^{1,p}(\varOmega;\R^{3\times3})}^{1/2}$.
Thus we have 
\begin{align}\label{interpolation-nabla-v-strongly}
\nabla\ovT\to\nabla\vv\ \ \ \text{ strongly in }\ L_{\rm w}^p(I;L^\infty(\varOmega;\R^{3\times3}))
\,.
\end{align}

From the assumed convexity of $\zetap(\chi,\cdot)$ and thus also
the convexity of $\LL\mapsto\zetap(\chi,\Fedev^{-1}\LL\Fedev)$,
its differential $\LL\mapsto\Fedev^{-\top}[\zetap(\chi,\cdot)]'(\Fedev^{-1}\LL\Fedev)\Fedev^\top$ is monotone. As also $\Ld\mapsto\DD_1$ is (even
strongly) monotone, by using \eq{EUL-L-visco-thermo-PT+4disc+}
with the boundary condition $(\ovT\Cdot\nabla)\oLpT=\bm0$ coming
from \eq{enhanced-BC}, we have 
\begin{align}\nonumber
&c_0\|\oLdT-\Ld\|_{L^2(I\times\varOmega;\R^{3\times3})}^2
+c_q\|\nabla\oLdT-\nabla\Ld\|_{L^q(I\times\varOmega;\R^{3\times3\times3})}^q
\\[-.2em]&\hspace*{1em}\nonumber\le\int_0^T\!\!\!\int_\varOmega\bigg(
\bbD_1(\oLdT{-}\Ld)\Colon(\oLdT{-}\Ld)
+\big(\HYPER_1(\nabla\oLdT){-}\HYPER_1(\nabla\Ld)\big)\Vdots
\nabla(\oLdT{-}\Ld)
\\[-.2em]&\hspace*{3.5em}\nonumber
+\bigg(
\big[\zetap(\uxT,\cdot)\big]'
\Big(\uFedevT^{\!\!\!-1}\DEV\big(\nabla\ovT{-}\oLdT\big)\uFedevT\Big)
\\[-.2em]&\hspace*{3.5em}\nonumber
-\big[\zetap(\uxT,\cdot)\big]'
\Big(\uFedevT^{\!\!\!-1}\DEV\big(\nabla\ovT{-}\Ld\big)\uFedevT\Big)\bigg)
\Colon\Big(\uFedevT^{\!\!\!-1}(\oLdT{-}\Ld)\uFedevT\Big)\bigg)\,\d\xx\d t
\\[-.2em]&\hspace*{1em}\nonumber=\int_0^T\!\!\!\int_\varOmega\!\!\bigg(
\big[\zetap(\uxT,\cdot)\big]'
\Big(\uFedevT^{\!\!\!-1}\DEV\big(\nabla\ovT{-}\Ld\big)\uFedevT\Big)
\Colon\Big(\uFedevT^{\!\!\!-1}(\Ld{-}\oLdT)\uFedevT\Big)
\\[-.1em]&\hspace*{1.5em}\nonumber
-\bigg(\frac{\uFedevT^{\!\!\!-\top}\DEV\boldM(\uFedevT)\uFedevT^{\!\!\!\top}}
{1{+}\big(\|\wt\ueT\|_{L^1(\varOmega)}{-}\THRESHOLD\big)^+\!}\,
+\bbD_1\Ld\bigg)\Colon(\oLdT{-}\Ld)
-\HYPER_1(\nabla\Ld)\Vdots\nabla(\oLdT{-}\Ld)\bigg)\,\d\xx\d t\to0
\\[-1.1em]\label{EUL-L-converge-strong-Ld-}\end{align}
with some $c_0,c_q>0$ related to \eq{EUL-L-Jeff-St-ass-D-D}. In
\eq{EUL-L-converge-strong-Ld}, we used 
\eq{EUL-L-converge-Ld-weak}, \eq{EUL-L-converge-strong-chi},
\eq{interpolation-nabla-v-strongly},
and \eq{EUL-L-converge-G-strong}, together with the continuity and growth
of $(\theta,\chi,\LL)\mapsto[\zetap(\chi,\cdot)]'(\LL)$,
as assumed in \eq{EUL-L-Jeff-St-thermo-ass-zetap}. From this, we obtain
\begin{align}
  \oLdT\to\Ld\ \ \ \text{ strongly in }\ L^2(I{\times}\varOmega;\Rtr)\,\cap\,
  L^q(I;W^{1,q}(\varOmega;\Rtr))\,.
\label{EUL-L-converge-strong-Ld}\end{align}
From the latter equation in \eq{EUL-L-visco-thermo-PT+2disc+}
written as $\oLpT=\uFedevT^{\!\!\!-1}\DEV(\nabla\ovT{-}\oLdT)\uFedevT$
and from $\wb\DD_{1,\tau}$ in \eq{EUL-L-visco-thermo-PT+1disc+},
we also obtain, with $s$ as in \eq{interpolation-nabla-v-strongly},
\begin{subequations}\begin{align}
&\oLpT\to\Lp\ \ \ \text{ strongly in }\ L^2(I;L^s(\varOmega;\Rtr))\ \ \ \text{ and}
\label{EUL-L-converge-strong-Lp}
\\&\wb\DD_{1,\tau}\to\DD_1\ \ \ \text{ strongly in }\
L^{q'}(I;W^{1,q}(\varOmega;\R^{3\times3})^*)\,.
\label{EUL-L-converge-strong-D1}\end{align}
\end{subequations}

This already allows for passing to the limit in the
continuity equation \eq{EUL-L-visco-thermo-PT+0disc+}
and the equations (\ref{EUL-L-visco-thermo-PT+disc+}c,d).
The limit passage
in the quasilinear momentum equation 
\eq{EUL-L-visco-thermo-PT+1disc+}
is a bit more technical. Using the monotonicity of the operator
$\vv\mapsto{\rm div}({\rm div}(\HYPER_0|\nabla\strain(\vv)|^{p-2}\nabla\strain(\vv))
-\bbD_0\strain(\vv))$ with the boundary conditions 
in \eq{enhanced-BC}
and using \eq{EUL-L-visco-thermo-PT+1disc+} tested by $\ovT{-}\wt\vv$,
we obtain
\begin{align}\nonumber
\!\!\Big(&\inf_{|E|=1}{\bbD}E\Colon E\Big)
\|\strain(\ovT{-}\wt\vv)\|_{L^2(I\times\varOmega;\R^{3\times3})}^2
+\mu c_p\|\nabla\strain(\vv)(\ovT{-}\wt\vv)\|_{L^p(I\times\varOmega;\R^{3\times3\times3})}^p
\\\nonumber\hspace*{0em}
\!\!&\le\!\int_0^T\!\!\Bigg(\!\!\int_\varOmega\!\bigg(\bbD\strain(\ovT{-}\wt\vv)
\Colon\strain(\ovT{-}\wt\vv)
\\[-.6em]&\nonumber\hspace{1em}
 +\HYPER_0\Big(|\nabla\strain(\ovT)|^{p-2}\nabla\strain(\ovT)
-|\nabla\strain(\wt\vv)|^{p-2}\nabla\strain(\wt\vv)\Big)\Vdots\nabla\strain(\ovT{-}\wt\vv)
\bigg)\,\d\xx+\big\langle\wb\DD_{1,\tau},\nabla(\ovT{-}\wt\vv)\big\rangle
\Bigg)\d t
 \\[-.6em]&=\nonumber
 \int_0^T\!\!\Bigg(\!\int_\varOmega\bigg(
 \Big(\orT\overline\GRAVITY_{\tau}-\pdt{\pp_\TAU}\Big)
 \Cdot(\ovT{-}\wt\vv)
 -\Big(\uTT{-}\opT{\otimes}\ovT\Big)\Colon\strain(\ovT{-}\wt\vv)
\\[-.6em]&\nonumber\hspace{1em}
-\bbD_0\strain(\wt\vv)\Colon\strain(\ovT{-}\wt\vv)
-\HYPER_0|\nabla\strain(\wt\vv)|^{p-2}\nabla\strain(\wt\vv)\Vdots\nabla\strain(\ovT{-}\wt\vv)
\bigg)\,\d\xx+\big\langle\wb\DD_{1,\tau},\nabla(\ovT{-}\wt\vv)\big\rangle
\Bigg)\d t
\\[-.5em]&\le\nonumber
 \int_0^T\!\!\Bigg(\!\int_\varOmega\bigg(
\orT\overline\GRAVITY_{\tau}\Cdot(\ovT{-}\wt\vv)
 +\pdt{\pp_\TAU}\Cdot\wt\vv
-(\opT{\otimes}\ovT)\Colon\strain(\wt\vv)
-\big(\bbD_0\strain(\wt\vv){+}
\uTT\big)\Colon\strain(\ovT{-}\wt\vv)
\\[-.8em]&\hspace{.2em}
-\HYPER_0|\nabla\strain(\wt\vv)|^{p-2}\nabla\strain(\wt\vv)\Vdots\nabla\strain(\ovT{-}\wt\vv)
\bigg)\,\d\xx+\big\langle\wb\DD_{1,\tau},\nabla(\ovT{-}\wt\vv)\big\rangle\Bigg)\d t+
\int_\varOmega\frac{|\pp_0|^2}{2\varrho_0}-\frac{|\pp_\TAU(T)|^2}{2\varrho_\TAU(T)}\,\d\xx
\nonumber\\[-.8em]
\label{EUL-L-strong+}\end{align}
with $\uTT=(\boldT(\uJT,\uFedevT)+\ADI(\uJT,\utT)\bbI)/(1+(\wt\ENG(\uJT,\uFedevT,\utT)-\THRESHOLD)^+)$
for any $\wt\vv\in L^p(I;W^{2,p}(\varOmega;\R^3))$ and with 
$\langle\cdot,\cdot\rangle$ denoting here the duality between
$W^{1,q}(\varOmega;\R^{3\times3})^*$ and $W^{1,q}(\varOmega;\R^{3\times3})$.
The last inequality in \eq{EUL-L-strong+} has again exploited the
convexity of the kinetic energy $(\pp,\varrho)\mapsto\frac12|\pp|^2/\varrho$
in the calculus:
\begin{align}\nonumber
\int_\varOmega\frac{|\pp_\TAU(T)|^2\!}{2\varrho_\TAU(T)}
-\frac{|\pp_0|^2}{2\varrho_0}\,\d\xx
&\le\int_0^T\!\!\!\int_\varOmega\!\pdt{\pp_\TAU\!}\Cdot\ovT
-\frac{|\ovT|^2\!}2\,\pdt{\varrho_\TAU}\,\d\xx\d t
\\&=
\int_0^T\!\!\!\int_\varOmega\!\pdt{\pp_\TAU\!}\Cdot\ovT
+\ovT\Cdot{\rm div}\big(\opT{\otimes}\ovT\big)
\,\d\xx\d t\,,
\label{EUL-L-strong-calc}
\end{align}
exploited the Green formula and the boundary condition $\ovT\Cdot\nn=0$.

Now we want to pass to the limit in \eq{EUL-L-strong+} or, more precisely,
to estimate the limit superior from above. For this, we again use
convexity of the kinetic energy, which causes the weak lower semicontinuity of
$(\varrho,\pp)\mapsto\int_\varOmega|\pp|^2/\varrho\,\d\xx$ as a convex functional
$\{\rho\in L^1(\varOmega);\,\rho\ge0\}\times L^2(\varOmega;\R^3)\to[0,+\infty]$.
Here we rely also on that $|\pp_\TAU(T)|^2/\varrho_\TAU(T)$ is bounded in
$L^1(\varOmega)$ due to \eq{est-of-p/sqrt-disc} and on that
$\varrho_\TAU(T)\to\varrho(T)$ even strongly in $C(\barOmega)$ due to
as, by the  Arzel\`a-Ascoli-theorem, $\varrho_\TAU\to\varrho$
strongly in $C(I{\times}\barOmega)$, and on that $\pp_\TAU(T)$ converges
weakly in $L^2(\varOmega;\R^3)$ due to \eq{est-of-p-v-disc}
to its limit which is $\pp(T)$ because
simultaneously $\pp_\TAU(T)\to\pp(T)$ weakly in $W^{2,p}(\varOmega;\R^3)^*$
due to \eq{EUL-L-converge-p}. For the term
$(\opT{\otimes}\ovT)\Colon\strain(\wt\vv)$, we use
\eq{EUL-L-converge-strong-p-} with \eq{Euler-small-converge-strong-vv}
which implies in particular that
$\opT{\otimes}\ovT\to\pp\,{\otimes}\,\vv$ even strongly in
$L^{p/2}(I{\times}\varOmega;\Rsym)$.
All this allows us to estimate of the limit superior of
\eq{EUL-L-strong+} from above:
\begin{align}
\nonumber
&\!\!\limsup_{\tau\to0}\bigg(\Big(\inf_{|E|=1}{\bbD}E\Colon E\Big)
\|\strain(\ovT{-}\wt\vv)\|_{L^2(I\times\varOmega;\R^{3\times3})}^2
+\mu c_p\|\nabla\strain(\ovT{-}\wt\vv)\|_{L^p(I\times\varOmega;\R^{3\times3\times3})}^p\bigg)
\\&\nonumber
\le\int_0^T\!\!\bigg\langle\pdt{\pp},\wt\vv\bigg\rangle
+\!\int_\varOmega\!\!\bigg(\!\varrho\GRAVITY\Cdot(\vv{-}\wt\vv)
-(\pp\otimes\vv)\Colon\strain(\wt\vv)
  -\big(\TT{+}\bbD_0\strain(\wt\vv)\big)
  \Colon\strain(\vv{-}\wt\vv)
  +\bbD_1\Le\Colon\nabla(\vv{-}\wt\vv)
\\[-.3em]&\nonumber\hspace{6.3em}
+\Big(\HYPER_1|\nabla\Ld|^{q-2}\nabla\Ld\!
{-}\HYPER_0|\nabla\strain(\wt\vv)|^{p-2}\nabla\strain(\wt\vv)\Big)
\bigg)\,\d\xx\d t+
\int_\varOmega\frac{|\pp_0|^2}{2\varrho_0}-\frac{|\pp(T)|^2}{2\varrho(T)}\,\d\xx
\\[-.4em]&=\nonumber
 \int_0^T\!\!\bigg\langle\pdt{\pp},\wt\vv-\vv\bigg\rangle
 +\int_\varOmega\!\bigg(\,\varrho\GRAVITY\Cdot(\vv{-}\wt\vv)
-\big(\TT{-}\pp\otimes\vv\big)\Colon\strain(\vv{-}\wt\vv)
+\big(\bbD_1\Le\!{-}\bbD_0\strain(\wt\vv)\big)\Colon
\nabla(\vv{-}\wt\vv)
\\[-.3em]&\hspace{8.3em}
+\Big(\HYPER_1|\nabla\Ld|^{q-2}\nabla\Ld\!
{-}\HYPER_0|\nabla\strain(\wt\vv)|^{p-2}\nabla\strain(\wt\vv)\Big)\Vdots\nabla\strain(\vv{-}\wt\vv)
\bigg)\,\d\xx\d t,\!\!\!\!
\label{EUL-L-strong+++}\end{align}
where $\langle\cdot,\cdot\rangle$ denotes here the duality between
$W^{2,p}(\varOmega;\R^3)^*$ and $W^{2,p}(\varOmega;\R^3)$ and where, for
the last equality, we used the calculus like \eq{EUL-L-strong-calc}
but for the continuous-in-time limit which holds as an equality. Choosing
$\wt\vv=\vv$ and reminding also \eq{Euler-small-converge-strong-vv}, we obtained
\begin{align}\label{Euler-small-converge-bar-strong-v}
\ovT\to\vv\ \ \text{ strongly in }\ L^p(I;W^{2,p}(\varOmega;\R^3))\,.
\end{align}
Thus, we can 
make the limit passage in
the momentum equation \eq{EUL-L-visco-thermo-PT+1disc+}
in the weak sense \eq{def-thermo-Ch5-momentum}.
Also, the limit passage in the heat equation \eq{EUL-L-visco-thermo-PT+3disc+}
is not possible. In particular, for convergence of the dissipative heat,
we exploit the strong convergences \eq{EUL-L-converge-strong-Ld},
\eq{EUL-L-converge-strong-Lp}, and \eq{Euler-small-converge-bar-strong-v}.

These convergences already allow us to pass to the limit in the system
\eq{EUL-L-visco-thermo-PT+disc+}. In the limit, we thus obtain a weak
solution to the system
\begin{subequations}\label{EUL-L-visco-thermo-PT+EPS}
\begin{align}\label{EUL-L-visco-thermo-PT+0EPS}
&\!\!\pdt{\varrho}=-\,{\rm div}\,\pp
\ \ \ \text{ with }\ \,\pp=\varrho\vv\,,\!
\\[-.2em]&\nonumber
\!\!\pdt{\pp}=
\frac{{\rm div}\boldT(J,\Fedev)+\nabla\ADI(J,\theta)}{1+(\|\wt\ENG(J,\Fedev,\theta)\|_{L^1(\varOmega)}{-}\THRESHOLD)^+}
+{\rm div}\big(\DD_{1}+\DD_{0}
-\pp{\otimes}\vv\big)+\varrho\GRAVITY
\\[-.3em]&\nonumber
\hspace{12em}\text{with }\
\DD_{1}=\bbD_1\Ld-{\rm div}\big(\HYPER_1|\nabla\Ld|^{q-2}\nabla\Ld\big)
\\[-.4em]&
\hspace{12em}\text{and }\ \ \DD_{0}=\bbD_0\strain(\vv)
-{\rm div}\big(\HYPER_0\big|\nabla\strain(\vv)\big|^{p-2}\nabla\strain(\vv)\big)\,,
\label{EUL-L-visco-thermo-PT+1EPS}
\\[-.1em]&
\!\big[\zetap(\chi,\cdot)\big]'(\Lp)
=\frac{\DEV\boldM(\Fedev)}{1+(\|\wt\ENG(J,\Fedev,\theta)\|_{L^1(\varOmega)}{-}\THRESHOLD)^+}
+\Fedev^\top\DD_{1}\Fedev^{-\top},\label{EUL-L-visco-thermo-PT+4EPS}
\\[-.2em]
&\!\!\pdt{\Fedev}\,=\Ld\Fedev-(\vv\Cdot\nabla)\Fedev
\ \text{ with }\ \ \Ld=\DEV(\nabla\vv)-\Fedev\Lp\Fedev^{-1},\!\!\!
\label{EUL-L-visco-thermo-PT+2EPS}
\\[-.2em]&\nonumber
\!\!\pdt{\Ent}
=
\frac{\ADI(J,\theta)\,{\rm div}\,\vv}
{1+(\|\wt\ENG(J,\Fedev,\theta)\|_{L^1(\varOmega)}{-}\THRESHOLD)^+}
+{\rm div}\big(\kappa(\chi)\nabla\theta-\Ent\vv\big)
+\xi(\chi;\vv,\Ld,\Lp)\ \ \ 
\\[-.7em]&
\hspace{15.5em}
\text{with }\ \
\Ent
=\wt\ENT(J,\theta)+\varrho\ell\chi
\ \ \text{ and}\ J=\rhoR/\varrho\,,
\label{EUL-L-visco-thermo-PT+3EPS}
\\[-.5em]&\!\!
\pdt{\chi}+H^{-1}(\chi)\ni
\Upsilon(\theta{-}\theta_\text{\sc pt})/\NU-\vv\Cdot\nabla\chi\,.
\label{EUL-L-visco-thermo-PT+5EPS}
\end{align}\end{subequations}

\medskip{\it Step 4: Elimination of the truncation}.
Considering the initial conditions \eq{EUL-L-PT-IC} and the
prescribed acceleration $\GRAVITY$ satisfying \eq{EUL-L-Jeff-St-thermo-ass-IC}
and \eq{EUL-L-Jeff-St-thermo-ass-load} and reminding the choice of
$\THRESHOLD$ related to each weak solutions these initial conditions,
we can see, through the $L^\infty$-estimates (\ref{est-of-basic+}a-c,e),
and \eq{est-of-basic-theta} of $J=\rhoR/\varrho$, $\Fedev$, and $\theta$, that
$\|\wt\ENG(\rhoR/\varrho,\Fedev,\theta)\|_{L^1(\varOmega)}^{}<\THRESHOLD$
and thus we can see that the truncation \eq{EUL-L-visco-thermo-PT+EPS}
is inactive at least on a short time interval. Yet,
these $L^\infty$-estimates allow for continuation to the
whole time interval $I=[0,T]$.
Thus the truncation in \eq{EUL-L-visco-thermo-PT+EPS}
is not active for the solutions obtained in Step~3 so that
this solution solve (in the weak sense) the original problem.
\hfill$\Box$

\begin{remark}[The role of inertia.]\upshape
Note that, in \eq{EUL-L-converge-strong-Ld-}, we needed the
strong convergnce of velocity gradient for which
the inerial term guaranteeing \eq{Euler-small-converge-strong-vv}
and then the interpolation \eq{interpolation-nabla-v-strongly}
was used. Only after this, we were able to prove the
strong convergence of $\nabla^2\ovT$. In the quasistatic variant
of this model, this argumentation would be corrupted. This
differs from the usual Maxwellian rheology (as in an incomplete
melting) where the quasistatic variant is well possible and
even simplifies the analysis considerably.
\end{remark}

\begin{remark}[Sharp phase interfaces.]\upshape
The estimate \eq{est-of-derivatives-chi-1} allows for the $W^{1,1}$-regularity 
of the phase fraction $\chi$ provided the initial
condition $\chi_0\in W^{1,1}(\varOmega)$. Analytically,
the information about $\nabla\chi$ is needed for compactness
of $\chi$ 
and the strong convergence \eq{EUL-L-converge-strong-chi},
which would hold even for the BV$(\varOmega)$-limit.
Such generalization assuming $\chi_0\in{\rm BV}(\varOmega)$
makes possible for sharp interfaces between the
solid and the liquid parts, as it was the original 
idea behind Stefan-type problems.
\end{remark}

\begin{remark}[A full time discretization]\label{rem-full-discret}\upshape
The kinematic flow rule \eq{EUL-L-Jeff-St4-thermo-grad} for $\Fedev$ can be
replaced by
\begin{align}\label{exponential-transformation}
\DT{\bm{G}}={\rm e}^{-\bm{G}}\Lp{\rm e}^{\bm{G}}
\end{align}
and then put $\Fedev={\rm e}^{\bm{G}}$. 
It is noteworthy that $\det\Fedev=1$ if and only if ${\rm tr}\bm{G}=0$.
This can be advantageously used to discretize in time directly
\eq{exponential-transformation}, i.e.
\begin{align}\label{exponential-transformation+}
  \frac{\bm{G}_\tau^k-\bm{G}_\tau^{k-1}}\tau=
       {\rm e}^{-\bm{G}_\tau^k}\LpTAU^k{\rm e}^{\bm{G}_\tau^k}-
       (\vv_\tau^k\cdot\nabla)\bm{G}_\tau^k\,,
\end{align}
which then preserve the discrete $\bm{G}$ traceless and thus
$\FedevTAU^{\!\!k}={\rm e}^{\bm{G}_\tau^k}$ isochoric; here the matrix calculus
${\rm tr}(ABC)={\rm tr}(CAB)$ is used with ${\rm e}^A{\rm e}^B={\rm e}^{A+B}$
if $AB=BA$ and $e^{0}=\bbI$. On the other hand, evaluation (even approximately)
of ${\rm e}^{\bm{G}}$ needed for \eq{exponential-transformation+}
and for other occurrences in (\ref{EUL-L-visco-thermo-PT+disc}b-d,e) needs
some extra effort, so that such full time discretization is
similarly conceptual as the time semi-discretization
\eq{EUL-L-visco-thermo-PT+disc}.
\end{remark}

\begin{remark}[An intensive internal variable: 
    phase-field fracture in the solid.]\label{rem-fracture}\upshape
  The elastic res\-ponse in the shear part suggests
  to extend the model by involving the concept of the {\it damage} mechanics
  in the solid part or,
  more applicably, of the phase-field fracture. As usual, it used 
  an intensive scalar-valued internal variable, denoted here by $\alpha$
  and valued in $[0,1]$ with the (mathematical) convention
  that $\alpha=0$ means complete damage. In contrast
  to damage models in solids as used in engineering where
  damage can evolve unidirectionally without any possible
  healing, here healing should ultimately be allowed
  to avoid fracture remaining after melting in liquid.
  This can be modelled by a convex dissipation potential
  $\zeta(\chi,\cdot):\R\to\R^+$ which depends on the
  the phase fraction $\chi$. An example might be 
$\zeta(\chi,\DT\alpha)=K_0\DT\alpha^-+K_1(1{-}\chi)\DT\alpha^++\epsilon\DT\alpha^2$ with $K_0$ fracture toughness and
  $K_1$ large to prevent healing in the solid phase, while
  $\epsilon>0$ is a small coefficient reflecting possible rate
  dependence and makes the analysis easier.
Thus, any fracture created in solid is to be quite immediately
forgotten (healed) after melting (i.e.\ $\chi\searrow0$) towards liquid water
is completed.
   Considering $\alpha$ as an intensive variable, the flow rule
   should involve the convective derivative as
   \begin{align}&\partial\big[\zeta(\chi,\cdot)\big](\DT\alpha)
     =\psi_\alpha'(J,\Fedev,\theta,\alpha)+\varkappa\Delta\alpha\,.
   \label{damage}\end{align}
   The last term \eq{damage} controls a length-scale of cracks which is modelled
   by a {\it phase-field} approach to {\it fracture}.
   The prominent model thus augments
   the free energy $\psi=\psi(J,\Fedev,\theta,\alpha)$ by the terms like
$\frac1{2\varkappa}(1{-}\alpha)^2+\frac12\varkappa|\nabla\alpha|^2$. 
   In our convective variant, to comply with the energy balance,
   the Cauchy stress is to be then augmented by
   the {\it Korteweg}-type {\it stress}
 $\varkappa\nabla\alpha{\otimes}\nabla\alpha{-}\frac\varkappa2|\nabla\alpha|^2\bbI$.
   Cf.\ \cite{Roub23SPTC} for an analysis in the linearized convected
   variant with an incomplete melting only.
\end{remark}

\def\Conc{c}

\begin{remark}[An extensive internal variable: a diffusant content.]\upshape
  Another worthy enhancement of the basic model is by a 
 diffusant content, denoted by $\Conc\ge0$, considered as an extensive variable.
 The evolution covering effects as {\it diffusion}, {\it convection},
 {\it swelling}, or {\it squeezing}, is then governed by 
 \begin{align}\label{diffuse-eq}
&\DT\Conc={\rm div}(m\nabla\mu)-\Conc\,{\rm div}\,\vv
\ \ \ \text{ with }\ \ 
\mu=\psi_{\Conc}'(J,\Fedev,\Conc,\theta)
   \end{align}
 with $m=m(J,\Conc,\theta,\chi)$ denoting the mobility (diffusivity)
 and $\mu$ denoting the {\it chemical potential}. The diffusant flux $-m\nabla\mu$
 thus expresses the so-called {\it Fick's law}. A prominent application is
 for the salt solved in water. In the Earth's oceans, salinity is
 about 3.5\,wt\% while (equilibriated) salinity of ice is nearly zero.
 Therefore, it is a strong coupling of diffusion \eq{diffuse-eq}
 with the phase transition during freezing/melting. Likewise,
 some metals (e.g.\ iron) exhibit tendency to reduce the content
 of alloying elements (e.g.\ carbon, nickel, manganese, etc.)
 during solidification of the melt.
\end{remark}

\subsubsection*{Acknowledgement}
This research has been covered by the CSF project 26-21297S
and the institutional support RVO:61388998 (\v{C}R).
The author is thankful to Dirk Peschka for pointing out the
alternative in Remark~\ref{rem-full-discret}.


\end{document}